\documentclass{amsart}
\usepackage{amsmath,amssymb}
\usepackage{graphicx}
\usepackage{latexsym}
\usepackage{color}
\usepackage{amscd}
\usepackage[all]{xy}
\usepackage{enumerate}
\usepackage{subfigure}
\theoremstyle{plain}
\newtheorem{theorem}{Theorem}[section]
\newtheorem{corollary}[theorem]{Corollary}
\newtheorem{lemma}[theorem]{Lemma}
\newtheorem{proposition}[theorem]{Proposition}

\theoremstyle{definition}
\newtheorem{definition}[theorem]{Definition}

\newtheorem{remark}[theorem]{Remark}



\newcommand{\Thirdmove}{{Framing} conjugation}
\newcommand{\thirdmove}{{framing} conjugation}

\newcommand{\thirdmoves}{{framing} conjugations}

\newcommand{\CPb}{\overline{\mathbb{CP}}{}^{2}}

\newcommand{\C}{\mathbb{C}}
\newcommand{\R}{\mathbb{R}}

\newcommand{\Z}{\mathbb{Z}}

\newcommand{\K}{{\rm K3}}

\newcommand{\Pa}{\partial}

\def\Ker{\operatorname{Ker}}

\def\Diff{\operatorname{Diff}}

\def\id{\operatorname{id}}
\def\Mod{\operatorname{Mod}}
\def\Crit{\operatorname{Crit}}
\def\Push{\operatorname{Push}}

\begin{document}

\title[Hurwitz equivalence for Lefschetz fibrations and multisections]{Hurwitz equivalence for \\ Lefschetz fibrations and their multisections}

\author[R. \.{I}. Baykur]{R. \.{I}nan\c{c} Baykur}
\address{Department of Mathematics and Statistics, University of Massachusetts, Amherst, MA 01003-9305, USA}
\email{baykur@math.umass.edu}

\author[K. Hayano]{Kenta Hayano}
\address{Department of Mathematics, Graduate School of Science, Hokkaido University, Sapporo, Hokkaido 060-0810, Japan}
\email{k-hayano@math.sci.hokudai.ac.jp}

\begin{abstract}
In this article, we characterize isomorphism classes of Lefschetz fibrations with multisections via their monodromy factorizations. We prove that two Lefschetz fibrations with multisections are isomorphic if and only if their monodromy factorizations in the relevant mapping class groups are related to each other by a finite collection of modifications, which extend the well-known Hurwitz equivalence. This in particular characterizes isomorphism classes of Lefschetz pencils. We then show that, from simple relations in the mapping class groups, one can derive new (and old) examples of Lefschetz fibrations which cannot be written as fiber sums of blown-up pencils. 
\end{abstract}

\maketitle

\section{Introduction} 

Lefschetz fibrations became a central tool in differential geometry and topology following Donaldson's insight in the late 1990s \cite{Donaldson} that one can effectively explore the topological aspects of manifolds by studying smooth maps on them which have locally holomorphic character. Since then they have gained a prominent role in symplectic topology, where a striking balance between flexibility and rigidity makes it possible to carry out topological constructions and geometric obstructions at the same time. In dimension $4$, a firmer grasp on the quickly developing theory of Lefschetz fibrations is pursued via factorizations in mapping class groups of surfaces \cite{Kas_1980, Matsumoto3, GS}. 

Our goal is to better understand how surfaces in symplectic $4$--manifolds arise in this setting. Surfaces in $4$--manifolds play a crucial role in our exploration of the topology of smooth and symplectic $4$--manifolds; they help determining the homeomorphism type, and distinguishing the diffeomorphism type.  As observed by Donaldson and Smith \cite{Donaldson_Smith_2003}, any symplectic surface can be seen as a \textit{multisection} or an \textit{$n$--section} of a Lefschetz fibration, which is a branched surface intersecting all the fibers positively $n$ times. In \cite{BaykurHayano_multisection}, we initiated an extensive study of symplectic surfaces via factorizations in more elaborate mapping class groups of surfaces. The current article aims to add to this effort by refining and exploring the correspondence between symplectic surfaces in symplectic $4$--manifolds, multisections of Lefschetz fibrations, and positive factorizations in surface mapping class groups. 

An isomorphism between two Lefschetz fibrations $(X_i, f_i)$ with $n$--sections $S_i$, $i=1,2$, is given by a pair of orientation-preserving diffeomorphisms between the total spaces $X_i$ and the base $2$--spheres, which commute with $f_i$ and match $S_i$. 
The mapping class group $\Mod(\Sigma_g;\{s_1,\ldots,s_n\})$, which consists of orientation-preserving self-diffeomorphisms of the genus-$g$ surface $\Sigma_g$ preserving the set of marked points $\{s_1,\ldots,s_n\}$, is the host to the lift of any monodromy factorization of $(X_i, f_i)$ prescribed by an $n$--section $S_i$. There are simple modifications of monodromy factorizations that naturally arise from a few choices made in the process of extracting these factorizations from a given fibration ---which we describe and study in detail in the later sections. Two monodromy factorizations in $\Mod(\Sigma_g;\{s_1,\ldots,s_n\})$ will be called Hurwitz equivalent if one can be obtained from the other by applying a sequence of this finite set of simple modifications. 

Our main theorem in this article is the following:

\begin{theorem} \label{mainthm}
For $g, n \geq 1$ there exists a one-to-one correspondence
\begin{eqnarray*}
\left\{\begin{array}{c}
\mbox{Genus--$g$ Lefschetz fibrations}\\
\mbox{with $n$--sections,}\\
\mbox{up to isomorphism}\\
\end{array}\right\}
&
\longleftrightarrow
&
\left\{\begin{array}{c}
\mbox{Monodromy factorizations}\\
\mbox{in $\Mod(\Sigma_g;\{s_1,\ldots,s_n\})$,} \\
\mbox{up to Hurwitz equivalence }\\
\end{array}\right\}
\end{eqnarray*}

\end{theorem}

\smallskip
\noindent When $g\geq 2$ and $n=0$, i.e. when the multisections are omitted, the above result is classical, due to Kas \cite{Kas_1980} and Matsumoto \cite{Matsumoto3}.

Theorem~\ref{mainthm} is proved in Section~\ref{Sec:equivalence}, after a review of background results in Section~\ref{S:factorization multisection}. We also provide an extension of this correspondence to one between \emph{framed} multisections of Lefschetz fibrations and monodromy factorizations in framed mapping class groups introduced in \cite{BaykurHayano_multisection}. In turn, through their monodromy factorizations, we obtain a full characterization of Lefschetz \emph{pencils}, up to isomorphisms that can permute base points; see Corollary~\ref{T:equivalence pencils}. 

In the last section, we turn to an intriguing question regarding the diversity of Lefschetz fibrations versus that of pencils, which allows us to demonstrate how geometric, topological, and algebraic aspects of the theory of Lefschetz fibrations can be nicely brought together. In \cite{Stipsicz}, Stipsicz asked whether every Lefschetz fibration can be obtained as fiber sums of blown-up pencils; that is to say, whether Lefschetz pencils are the building blocks of all Lefschetz fibrations via blow-ups and fiber sums. In Section~\ref{Stipsicz}, we will illustrate a way to produce counter-examples, using the well-known $5$--chain relation in the mapping class group of a genus--$2$ surface, along with monodromy modifications involving multisections, and a variety of geometric and topological results packaged in a recipe from  \cite{BaykurHayano_multisection} we will be following here. We moreover show --by appealing to above Hurwitz equivalences-- that the only other counter-example we know of, a genus--$2$ Lefschetz fibration of Auroux (shown to be a counter-example by Sato in \cite{Sato_2008}), can also be derived from the same scheme.

\vspace{0.2in}
\noindent \textit{Acknowledgements.} The results of this article were partially presented by the second author at the 13th International Workshop on Real and Complex Singularities held in August 2014 at Sao Carlos, Brazil. He would like to thank the organizers of the workshop for the opportunity and for their invitation to prepare this manuscript. 
The authors would also like to thank the anonymous referee for helpful comments. The first author was partially supported by the Simons Foundation Grant $317732$. The second author was supported by JSPS and CAPES under the Japan-Brazil research cooperative program and JSPS KAKENHI (26800027).

\section{Multisections of Lefschetz fibrations via positive factorizations}\label{S:factorization multisection}

In this section we will briefly review the basic definitions and properties of Lefschetz fibrations and their multisections, focusing on how they can be captured and studied as certain factorizations in mapping class groups of surfaces. For a more detailed exposition, the reader can turn to \cite{GS} and \cite{BaykurHayano_multisection}.

Throughout the paper, all manifolds we work with are assumed to be closed, connected and oriented, unless otherwise noted. 

\subsection{Lefschetz pencils, fibrations, and multisections} \

A \emph{Lefschetz fibration} $(X,f)$ is a smooth map $f:X\to S^2$, from a $4$--manifold $X$ onto the $2$--sphere, which only has \emph{nodal singularities}, that is, for any $x$ in the \emph{critical locus} $\Crit(f)$, there exist orientation-preserving complex coordinate neighborhoods $(U,\varphi)$ at $x\in X$ and $(V,\psi)$ at $f(x)\in S^2 \cong \mathbb{CP}^1$, such that 
\[\psi\circ f \circ \varphi^{-1}(z,w) = z^2+w^2 \, . \] 
So $f$ is a submersion at all but finitely many points, where we have the local model of a complex nodal singularity. For $g$ the genus of a regular fiber, we call $(X,f)$ a \emph{genus--$g$ Lefschetz fibration}.

A \emph{Lefschetz pencil} $(X,f)$ is a Lefschetz fibration $f:X\setminus B\to S^2$, where $B$ is a \textit{non-empty} discrete set in $X$, called the \emph{base locus}, such that there exist an orientation-preserving complex coordinate neighborhood $(U,\varphi)$ around each \emph{base point} $x \in B$ and a diffeomorphism $\psi:S^2\to \mathbb{CP}^1$, which together satisfy
\[
\psi\circ f\circ \varphi^{-1}(z,w) = [z:w]. 
\]
We say $(X,f)$ is a \emph{genus--$g$ Lefschetz pencil with $n$ base points} for $g$ the genus of the regular fiber (compactified by adding the base points), and $n=|B|$. Given a genus--$g$ Lefschetz pencil $(X,f)$ with $n$ base points, we can obtain a genus--$g$ Lefschetz \emph{fibration} $f': X'=X \# n  \CPb \to S^2$ with $n$ disjoint sections $S_j$ of self-intersection $-1$, each arising as an exceptional sphere of the blow-up at the base point $x_j$. The correspondence is canonical, as one can blow-down all the $S_j$ in the fibration $(X',f')$ to recover the pencil $(X,f)$  (e.g. \cite[\S.8.1]{GS}). 

Recall that a \emph{symplectic structure} is a closed non-degenerate $2$--form $\omega$ on a smooth manifold, such as the K\"{a}hler form on a complex algebraic variety. A \emph{symplectic $4$--manifold} is then a pair $(X, \omega)$. The prominent role of Lefschetz fibrations in differential geometry and topology is mostly due to Donaldson's amazing result from the late 1990s, who showed that an analogue of the classical Lefschetz hyperplane theorem for complex algebraic surfaces holds in this more flexible setting: every \emph{symplectic} $4$--manifold admits a Lefschetz pencil \cite{Donaldson}. Conversely, generalizing an idea of Thurston, Gompf observed that every $4$--manifold admitting a Lefschetz pencil or a non-trivial (i.e.~with non-empty critical locus) Lefschetz fibration is symplectic \cite{GS}. Furthermore, one can strike a compatibility condition between the pairs $(X, \omega)$ and $(X, f)$, which asks for the fibers of $f$ to be symplectic surfaces with respect to $\omega$.

\begin{theorem}[Donaldson, Gompf] 
Every symplectic $4$--manifold $(X, \omega)$ admits a compatible Lefschetz pencil, and every Lefschetz pencil~/~non-trivial fibration $(X, f)$ can be equipped with a compatible symplectic form.
\end{theorem}

The main companion of a Lefschetz fibration $(X, f)$ in this paper will be an embedded surface $S$ which sits in $X$ in a rather special way with respect to $f$.

\begin{definition} [\cite{BaykurHayano_multisection}]\label{Def:multisection}
A (possibly disconnected) closed oriented surface $S\subset X$ is called a \emph{multisection}, or an \emph{$n$--section}, of a Lefschetz fibration $(X,f)$ if it satisfies the conditions:
\begin{enumerate}
\item
The restriction $f|_{S}$ is an $n$--fold simple branched covering, 
\item
The restriction of the differential $df_x:N_xS\to T_{f(x)}S^2$ preserves the orientation for any branched point $x\in S$ of $f|_{S}$, where $N_xS\subset T_xX$ is a normal space of $T_xS$, which has the canonical orientation derived from that of $X$, 
\item
For any branched point $x\in S\cap \Crit(f)$ of $f|_{S}$, there exist complex coordinate neighborhoods $(U,\varphi)$ and $(V,\psi)$ which make the following diagram commute: 
\[
\begin{CD}
(U,U\cap S) @>\varphi >> (\C^2,\C\times\{0\}) \\
@V f VV  @VV (z,w)\mapsto z^2+w^2 V \\
V @>\psi >> \C.  
\end{CD}
\]
\end{enumerate}
A triple $(X, f, S)$ will denote a Lefschetz fibration $(X,f)$ and its multisection $S$.
\end{definition}

Just like how a Lefschetz fibration locally behaves like a holomorphic map, a multisection behaves like a holomorphic curve, intersecting the fibers all positively, and so that the restriction of the fibration map to it is a holomorphic branched covering onto ${{\mathbb{CP}}{}^{1}}$. Multisections are found in abundance, as observed by Donaldson and Smith (who referred to them as \emph{standard surfaces}):  for any symplectic surface $S$ in a symplectic $4$--manifold $(X, \omega)$, there exists a compatible Lefschetz pencil $(X,f)$, such that $S$ is a multisection of $f|_{X \setminus B}$, and conversely, for any triple $(X,f, S)$, there exists a compatible symplectic form $\omega$ making the fibers and $S$ symplectic \cite{Donaldson_Smith_2003}.

\smallskip
\subsection{Local and global monodromies, positive factorizations} \

We will make a few additional assumptions on $(X,f,S)$, merely to simplify our upcoming discussions on how to describe Lefschetz fibrations and their multisections in terms of certain factorizations in surface mapping class groups. First, we will assume that $f$ is injective on $\Crit(f)$, and also that each branched point of a multisection, if not contained in $\Crit(f)$, does not lie on a singular fiber, i.e. not contained in $f^{-1}(\Crit(f))$. Both of these can be always achieved after a small perturbation. These assumptions will allow us to get standard local models for $(X,f,S)$ for the fibration over $S^2 \setminus f(\Crit(f) \cup \Crit(f|_S))$ --and not only \textit{around} the isolated points in $\Crit(f) \cup \Crit(f|_S)$. 

It is also customary to assume that $f$ is \emph{relatively minimal}, that is, no fiber contains a sphere with self-intersection $-1$, which otherwise could be blown-down without destroying the rest of the fibration. A Lefschetz \emph{pencil} $(X,f)$ is said to be \emph{relatively minimal}, if no fiber component is a self-intersection $-1+k$ sphere containing $k$ points of $B$, to ensure that the associated Lefschetz fibration $(X',f')$ is relatively minimal. As we will see shortly, this assumption is needed to guarantee that no information on the local topology is lost when we look at the monodromy of the fibration.

For a surface $\Sigma$ and points $s_1,\ldots,s_n\in \Sigma$, let $\Diff(\Sigma;\{s_1,\ldots,s_n\})$ be the group of orientation-preserving diffeomorphisms of $\Sigma$ which preserve the set $\{s_1,\ldots,s_n\}$. We call $\Mod(\Sigma;\{s_1,\ldots,s_n\})=\pi_0(\Diff(\Sigma;\{s_1\ldots,s_n\})$ the \emph{mapping class group of \, $\Sigma$ with marked points $\{s_1,\ldots, s_n\}$}. It consists of elements of $\Diff(\Sigma;\{s_1,\ldots,s_n\})$, modulo isotopies fixing the set $\{s_1,\ldots,s_n\}$, where the group structure is induced by compositions of maps, that is, $[\varphi_1]\cdot [\varphi_2] = [\varphi_1\circ \varphi_2]$ for $\varphi_1,\varphi_2 \in \Diff(\Sigma_g;\{s_1,\ldots,s_n\})$.

Let $\Sigma_g^n$ denote a genus--$g$ surface with $n$ boundary components, and take points $u_1,\ldots,u_n\in \Pa \Sigma_g^n$ which cover the elements of $\pi_0(\partial \Sigma_g^n)$. The \emph{framed mapping class group} $\Mod(\Sigma_g^n;\{u_1,\ldots,u_n\})$ consists of orientation-preserving diffeomorphisms of $\Sigma_g^n$ which preserve set of marked points $\{u_1,\ldots,u_n\}$, modulo isotopies fixing the same data \cite{BaykurHayano_multisection}. Clearly, one can pass to a closed surface $\Sigma_g = \Sigma_g^0$ by capping the boundaries by disks, the centers of which we label as $s_1,\ldots,s_n\in \Sigma_g$. This boundary capping map induces a surjective homomorphism 
\[ \Mod(\Sigma_g^n;\{u_1, \ldots,u_n\})\twoheadrightarrow \Mod(\Sigma_g;\{s_1,\ldots,s_n\}) \, .\]

Now let $(X,f)$ be a Lefschetz fibration with $l$ critical points, $S$ its $n$--section and $\Crit(f|_S)\setminus \Crit(f)=\{b_1,\ldots,b_k\}\subset X$ the set of branched points of $f|_{S}$ away from $\Crit(f)$. 
Set $f(\Crit(f)\cup \Crit(f|_S)) = \{a_1,\ldots,a_{k+l}\}$, and take paths $\alpha_1,\ldots,\alpha_{k+l}\subset S^2$ with a common initial point  $p_0 \in S^2\setminus f(\Crit(f)\cup \Crit(f|_S))$ such that
\begin{itemize}
\item
$\alpha_1,\ldots,\alpha_{k+l}$ are mutually disjoint except at $p_0$, 
\item
$\alpha_i$ connects $p_0$ with $a_i$, 
\item
$\alpha_1,\ldots,\alpha_{k+l}$ are ordered counterclockwise around $p_0$, i.e.~there exists a small loop around $p_0$ oriented counterclockwise, hitting each $\alpha_i$ only once in the given order.
\end{itemize}  

\noindent
We take a loop $\widetilde{\alpha}_i$ with the base point $p_0$ by connecting $\alpha_i$ with a small counterclockwise circle with center $a_i$. Let $\mathcal{H}$ be a horizontal distribution of $f|_{X\setminus (\Crit(f)\cup \Crit(f|_S))}$, that is, $\mathcal{H} = \{\mathcal{H}_x\}_{x\in X\setminus (\Crit(f)\cup \Crit(f|_S))}$ is a plane field such that $\Ker(df_x) \oplus \mathcal{H}_x = T_xX$ for any $x\in X\setminus (\Crit(f)\cup \Crit(f|_S))$. 
We assume that $\mathcal{H}_x = T_xS$ for any $x\in S\setminus \Crit(f|_S)$. 
Using $\mathcal{H}$, we can take a lift of the direction vector field of $\widetilde{\alpha}_i$ and a flow of this lift gives rise to a self-diffeomorphism of $f^{-1}(p_0)$. 
We call this diffeomorphism a \emph{parallel transport} of $\widetilde{\alpha}_i$ and its isotopy class a \emph{local monodromy} around $a_i$. 
Note that a local monodromy does not depend on the choice of $\mathcal{H}$. 
Indeed, for any horizontal distribution $\mathcal{H}$ we can find a Riemannian metric $g$ such that $\mathcal{H}_x$ is equal to $(\Ker(df_x))^\perp$, in particular for any two horizontal distribution there exists a one-parameter family of horizontal distributions connecting the given two distributions.

Under an identification of the pair $(f^{-1}(p_0),f^{-1}(p_0)\cap S)$ with $(\Sigma_g,\{s_1,\ldots,s_n\})$, we can regard a parallel transport as a diffeomorphism in $\Diff(\Sigma_g;\{s_1,\ldots,s_n\})$, and thus, a local monodromy as a mapping class in $\Mod(\Sigma_g;\{s_1,\ldots,s_n\})$. 
We denote this mapping class by $\xi_i$. 
Since the concatenation $\widetilde{\alpha}_1\cdots \widetilde{\alpha}_{k+l}$ is null-homotopic in $S^2\setminus f(\Crit(f)\cup \Crit(f|_S))$, the composition $\xi_{k+l}\cdots \xi_{1}$ is the unit element of the group $\Mod(\Sigma_g;\{s_1,\ldots,s_n\})$. 
The factorization 
\[\xi_{k+l}\cdots \xi_{1} =1 \ \ \ \ \text{in $\Mod(\Sigma_g;\{s_1,\ldots,s_n\})$} \]
is called a \emph{monodromy factorization of the triple $(X,f,S)$}, which we will denote in short by $W_{X,f,S}$ (as a \emph{word} in $\xi_i \in \Mod(\Sigma_g;\{s_1,\ldots,s_n\}$). 

Analyzing the local models around $\Crit(f)$ and $\Crit(f|_S)$ (which, remember, might contain in common the type (3) branched points in Definition~\ref{Def:multisection}), we identify three standard elements in the mapping class group  $\Mod(\Sigma_g;\{s_1,\ldots,s_n\})$ that appear as a factor $\xi_i$ above \cite{BaykurHayano_multisection}: If the fiber $f^{-1}(a_i)$ contains a Lefschetz critical point which is not a branched point of $f|_{S}$, then $\xi_i$ is a right-handed Dehn twist along some simple closed curve $c\subset \Sigma_g\setminus \{s_1,\ldots,s_n\}$, which is called a \emph{vanishing cycle} of a Lefschetz critical point in $f^{-1}(a_i)$ \cite{Kas_1980}. (Relative minimality of $(X,f)$ now guarantees that $c$ is not null-homotopic, so we do not have a ``hidden'' Dehn twist factor.)  If $f^{-1}(a_i)$ contains a branched point of $S$ away from $\Crit(f)$, then $\xi_i$ is a half twist $\tau_{\gamma}$ along some path $\gamma\subset \Sigma_g$ between some $s_j$ and $s_{j'}$. Lastly, if $a_i$ is the image of a point in $\Crit(f) \cap \Crit(f|_S)$, we get a mapping class $\xi_i$ which is represented by a Dehn twist $\widetilde{t_{c}}$ shown in Figure~\ref{F:lift dehntwist} for some simple closed curve $c\subset \Sigma_g$ going through $s_j$ and $s_{j'}$.  

\begin{figure}[htbp]
\centering
\includegraphics[width=60mm]{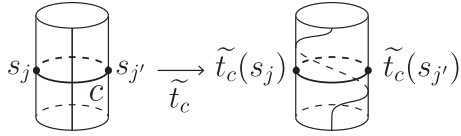}
\caption{A lift of a Dehn twist. }
\label{F:lift dehntwist}
\end{figure}

\noindent Observe that under the forgetful homomorphism, only $\xi_i$ that are Dehn twists (possibly going through $\Crit(f|_S)$) survive, yielding the standard monodromy factorization $W_{X,f}$ of $(X,f)$ of the form $t_{c_l} \cdots t_{c_1} = 1$ in $\Mod(\Sigma_g)$. In other words, the factorization $W_{X,f,S}$ is a \emph{lift} of the factorization $W_{X,f}$. 

It is worth noting that each standard element $\xi_i$ discussed above, let it be a Dehn twist or an arc twist, comes with a preferred orientation, corresponding to positive (right-handed) Dehn twists and arc/braid twists. Any factorization $\xi_{k+l}\cdots \xi_{1} =1$ in $\Mod(\Sigma_g;\{s_1,\ldots,s_n\})$, which consists of only these three types of elements is called a \emph{positive factorization} (of the identity) in $\Mod(\Sigma_g;\{s_1,\ldots,s_n\})$, and it conversely gives rise to a triple $(X,f,S)$. 

We can summarize these as follows (which is a direct corollary of \cite[Theorem~1.1]{BaykurHayano_multisection} obtained by applying the boundary capping homomorphism to the framed mapping class group):

\begin{theorem} \cite{BaykurHayano_multisection} \label{MainQuote}
Let $(X, f, S)$ be a genus--$g$ Lefschetz fibration with $l$ critical points, where $S$ is a connected $n$--section which has $k$ branched points away from $\Crit(f)$ and $r$ branched points on $\Crit(f)$. Then $(X,f,S)$ has a monodromy factorization $W_{X,f,S}$ of the form \, $\xi_{k+l}\cdots \xi_{1} =1$ \, in $\Mod(\Sigma_g;\{s_1,\ldots,s_n\})$, where among $\xi_i$ $k$ many are half-twists $\tau_{\gamma_i}$, $r$ many are Dehn twists $\widetilde{t_{c_{i}}}$ through two marked points in $\{s_1,\ldots,s_n\}$, and the rest are Dehn twists along curves missing the marked points. 
Moreover, $g(S)=  \frac{1}{2}(k+r) - n + 1$ and the union $\{s_1,\ldots,s_n\}\cup \Gamma \cup C$ is connected, where $\Gamma$ is the union of paths between points in $\{s_1,\ldots,s_n\}$ corresponding half twists in the factorization and $C$ is the union of simple closed curves going through two points in $\{s_1,\ldots,s_n\}$ corresponding lifts of Dehn twists in the factorization. 

Conversely, from any such positive factorization of $1$ in $\Mod(\Sigma_g;\{s_1,\ldots,s_n\})$, subject to conditions listed above, one can construct a genus--$g$ Lefschetz fibration $(X,f)$ with $l$ vanishing cycles $c_1, \ldots, c_l$, and a connected $n$--section $S$ of genus $g(S)=\frac{1}{2}(k+r) - n + 1$.
\end{theorem}

\begin{remark}
	
We can modify Theorem \ref{MainQuote} so that it also holds for a \emph{disconnected} multisection: in this case the union $\{s_1,\ldots,s_n\}\cup \Gamma \cup C$ is not necessarily connected (the number of components of the union coincides with that of $S$), and the Euler characteristic $\chi(S)$ is equal to $2n-(k+r)$. 

\end{remark}

\smallskip
\begin{remark}\label{R:lift factorization}
As we have shown in \cite{BaykurHayano_multisection} the positive factorization $W_{X,f,S}$ of the identity element in $\Mod(\Sigma_g;\{s_1,\ldots,s_n\})$ lifts to another positive factorization of a product of boundary parallel Dehn twists in $\Mod(\Sigma_g^n;\{u_1,\ldots,u_n\})$. This geometrically corresponds to removing a framed tubular neighborhood of $S$. The latter positive factorization consists of standard factors $t_{c_i}$, which are Dehn twists in the interior of $\Sigma_g^n$, and the lifts of $\tau_{\gamma}$ and $\widetilde{t_c}$ as shown in Figure~\ref{F:lift localmonodromy}, both of which interchange the two boundary components in prescribed ways. 

\begin{figure}[htbp]
\centering
\includegraphics[height=24mm]{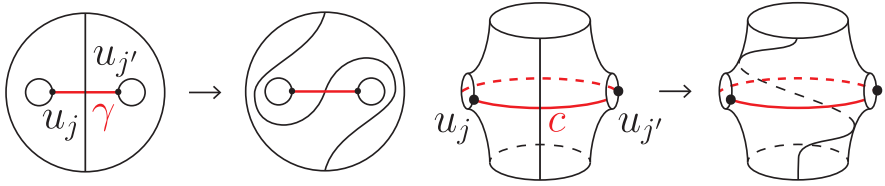}
\caption{Lifts of $\tau_{\gamma}$ and $\widetilde{t_c}$. }
\label{F:lift localmonodromy}
\end{figure}
For simplicity, we denote these lifts by $\tau_{\gamma}$ and $\widetilde{t_c}$ as well. This monodromy factorization in the framed mapping class group allows us to capture the self-intersection number of $S$. 
\end{remark}

\vspace{0.2in}
\section{Equivalence of Lefschetz fibrations with multisections}\label{Sec:equivalence}

The goal of this section is to establish a one-to-one correspondence between triples $(X,f,S)$, where $S$ is an $n$--section of a genus--$g$ Lefschetz fibration $(X,f)$, and positive factorizations $W_{X,f, S}$ of the form $\xi_{k+l}\cdots \xi_{1} =1$ in $\Mod(\Sigma_g;\{s_1,\ldots,s_n\})$, modulo natural equivalence relations on both sides, which we will spell out right away.

The triples $(X_i,f_i,S_i)$, $i=1,2$, are said to be \emph{equivalent} (or \emph{isomorphic}) if there exist orientation-preserving diffeomorphisms $\Phi:X_1\to X_2$ and $\phi:S^2\to S^2$ such that $\Phi(S_1) = S_2$ and $f_2\circ \Phi = \phi\circ f_1$. Clearly, a necessary condition for $(X_i, f_i, S_i)$ to be equivalent is that both fibrations $f_i$ should have the same genus $g$, and the multisections $S_i$ should have the same covering degree $n$.

As we noted in the previous section, a monodromy factorization $(X, f,S)$ does not depend on the choice of a horizontal distribution $\mathcal{H}$. It does however depend on the choice of paths $\alpha_1,\ldots,\alpha_{k+l}$ and that of an identification of $(f^{-1}(p_0),f^{-1}(p_0)\cap S)$ with $(\Sigma_g,\{s_1,\ldots,s_n\})$. Identical to the well-known case of a monodromy factorization of a \emph{pair} $(X,f)$ one can (see e.g. \cite{GS}) easily verify that two monodromy factorizations for a triple $(X,f,S)$, derived from different choices of paths and identifications can be related by successive applications of the following two types of modifications:
\begin{enumerate}
\item
\emph{Elementary transformation}, which changes a factorization as follows: 
\[
\xi_{k+l}\cdots \xi_{i+1}\xi_i\cdots \xi_1 \longleftrightarrow \xi_{k+l}\cdots(\xi_{i+1}\xi_{i}\xi_{i+1}^{-1})\xi_{i+1}\cdots \xi_1. 
\]

\noindent
Note that $(\xi_{i+1}\xi_{i}\xi_{i+1}^{-1})$ in the right hand side represents a single factor of the factorization. 

\item
\emph{Global conjugation}, which changes each member of a factorization by the conjugation of some mapping class $\psi\in \Mod(\Sigma_g;\{s_1,\ldots,s_n\})$: 
\[
\xi_{k+l}\cdots \xi_1 \longleftrightarrow (\psi\xi_{k+l}\psi^{-1})\cdots (\psi\xi_1\psi^{-1}). 
\]
\end{enumerate}
\noindent
We will thus call two factorizations of the unit element $1\in \Mod(\Sigma_g;\{s_1,\ldots,s_n\})$ \emph{Hurwitz equivalent} if one can be obtained from the other after a sequence of these two types of modifications.

\smallskip
\subsection{Equivalence of monodromy factorizations} \

This subsection will be devoted to the proof of the following theorem, which, together with 
Theorem~\ref{MainQuote}, implies the main result of our paper, Theorem~\ref{mainthm}. 

\begin{theorem}\label{T:equivalence multisection factorization}
Let $(X_i, f_i, S_i)$, $i=1,2$ be a genus--$g$ Lefschetz fibration with an $n$--section $S_i$. 
Suppose that $2-2g-n$ is negative, that is, $f_i^{-1}(p_0)\setminus (f_i^{-1}(p_0)\cap S_i)$ is a hyperbolic surface for a regular value $p_0$. The triples $(X_1,f_1,S_1)$ and $(X_2,f_2,S_2)$ are equivalent if and only if their monodromy factorizations $W_{X_1,f_1,S_1}$ and $W_{X_2,f_2,S_2}$ are Hurwitz equivalent. 
\end{theorem}

To prove Theorem~\ref{T:equivalence multisection factorization}, we will need a few preliminary results on mapping classes in  $\Mod(\Sigma_g;\{s_1,\ldots,s_n\})$. 

\begin{lemma}\label{T:isotopy equality half twist}
Let $\gamma_1, \gamma_2\subset \Sigma_g$ be simple paths between distinct marked points $s_i$ and $s_j$. Then $\tau_{\gamma_1}=\tau_{\gamma_2}$ in $\Mod(\Sigma_g;\{s_1,\ldots,s_n\})$ if and only if $\gamma_1$ and $\gamma_2$ are isotopic relative to the points $s_1,\ldots,s_n$. 

\end{lemma}

\begin{proof}
The ``if'' part is obvious.  To prove the ``only if'' part we assume that $\gamma_1$ and $\gamma_2$ are not isotopic and show that $\tau_{\gamma_1}$ and $\tau_{\gamma_2}$ are not equal. For simple curves $d_1, d_2$, we denote the geometric intersection number of $d_1$ and $d_2$ by $i(d_1,d_2)$, that is, $i(d_1,d_2)$ is the minimum number of intersections between two curves isotopic (relative to $s_1,\ldots,s_n$) to $d_1$ and $d_2$. 

Without loss of generality we may assume that $\gamma_1$ and $\gamma_2$ are in minimal position. 
Let $d$ be the boundary of a regular neighborhood of $\gamma_1$. 
The curve $d$ does not intersect $\gamma_1$.
On the other hand, $i(d,\gamma_2)$ is not equal to $0$.  
To see this, we will check that there is no bigon between sub-paths of $d$ and $\gamma_2$ (see the bigon criterion in \cite{Farb_Margalit_2011}). 
As shown in Figure~\ref{F:simple closed curve d} there are three types of regions which are candidates of such bigons. 
\begin{figure}[htbp]
\centering
\subfigure[]{\includegraphics[height=25mm]{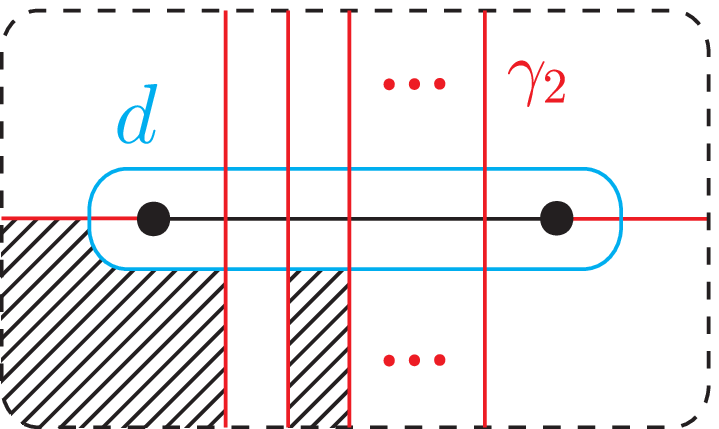}
\label{F:simple closed curve d1}}
\hspace{.8em}
\subfigure[]{\includegraphics[height=25mm]{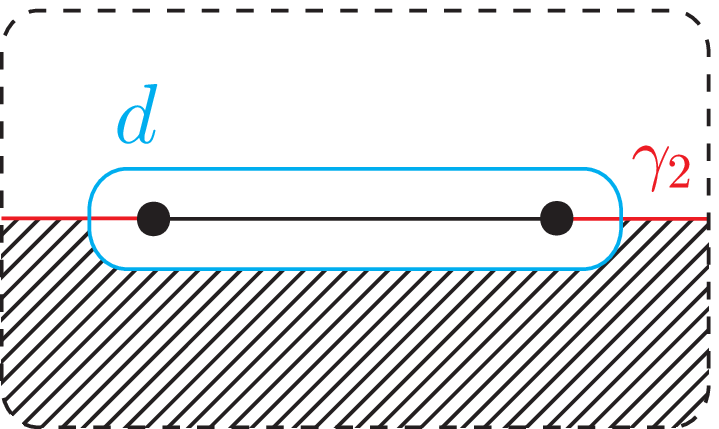}
\label{F:simple closed curve d2}}
\caption{Candidates of bigons around $d$.}
\label{F:simple closed curve d}
\end{figure}
The two shaded regions in Figure~\ref{F:simple closed curve d1} cannot be bigons since $\gamma_1$ and $\gamma_2$ are in minimal position. 
Since $\gamma_1$ and $\gamma_2$ are not isotopic, the shaded region in Figure~\ref{F:simple closed curve d2} is not a bigon. 

Since $i(d,\tau_{\gamma_1}(d))$ is equal to $0$, it is sufficient to prove that $i(d,\tau_{\gamma_2}(d))$ is not equal to $0$.  
Let $d'$ be the simple closed curve obtained by changing a parallel copy of $d$ around $\gamma_2$ as shown in Figure~\ref{F:result halftwist}. 
The curve $d'$ represents the isotopy class $\tau_{\gamma_2}(d)$. 
\begin{figure}[htbp]
\centering
\includegraphics[height=25mm]{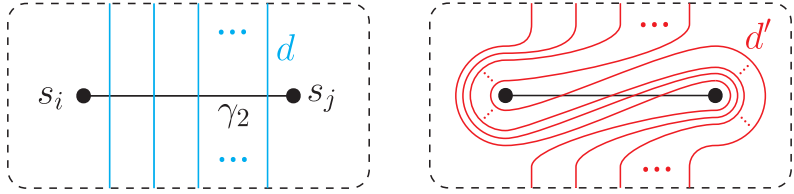}
\caption{The curves $d$ and $d'$ around $\gamma_2$. }
\label{F:result halftwist}
\end{figure}
It is easy to see that the number of the intersections between $d$ and $d'$ is equal to $2i(d,\gamma_2)^2$. 
In what follows, we will prove that $d$ and $d'$ are in minimal position using the bigon criterion. 
As shown in Figure~\ref{F:candidate bigon} there are six types of regions which are candidates of innermost bigons. 
\begin{figure}[htbp]
\centering
\subfigure[]{
\includegraphics[height=25mm]{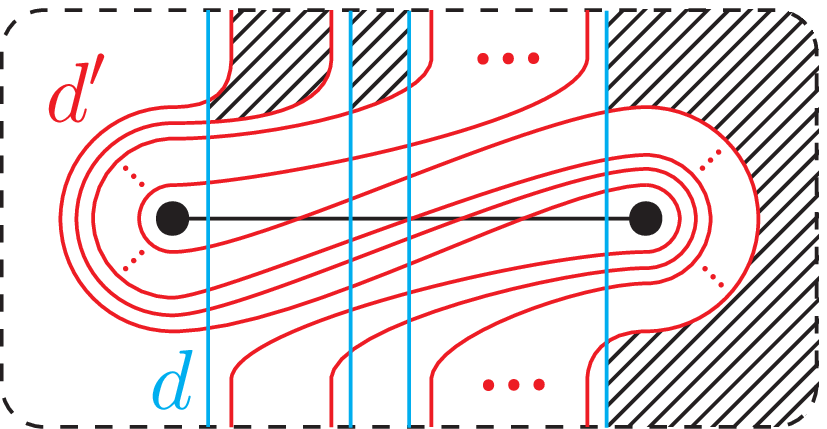}
\label{F:candidate bigon1}
}
\hspace{.8em}
\subfigure[]{
\includegraphics[height=25mm]{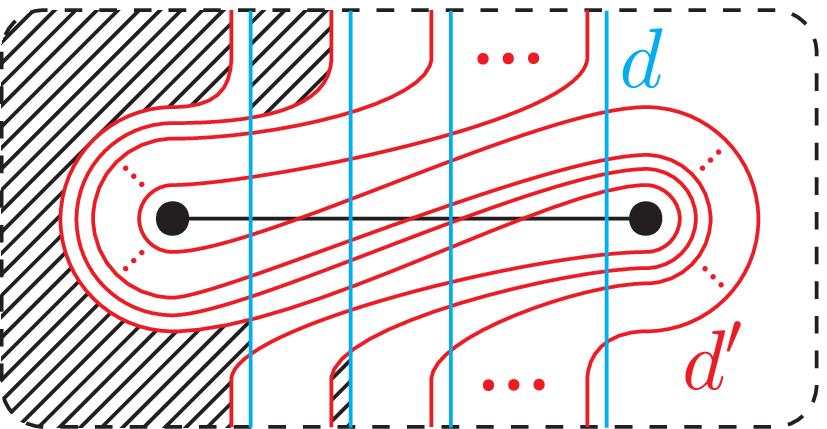}
\label{F:candidate bigon2}
}
\caption{Shaded regions are candidates of bigons. }
\label{F:candidate bigon}
\end{figure}
If the far right region in Figure~\ref{F:candidate bigon1} were a bigon, $d$ would be isotopic to a small circle with center $s_j$, but it is not the case since $d$ does not intersect $\gamma_1$. 
Similarly, we can also verify that the far left region in Figure~\ref{F:candidate bigon2} is not a bigon. 
If either of the rest of two regions in Figure~\ref{F:candidate bigon1} were a bigon, then $d$ and $\gamma_2$ would form a bigon, which contradicts the assumption that $d$ and $\gamma_2$ are in minimal position. 
As for the rest of two regions in Figure~\ref{F:candidate bigon2}, either the boundary of them contain at least two sub-paths of $d$, or they contain the point $s_i$ or $s_j$. 
In either case, these regions cannot be bigons. 
We can eventually conclude that $d'$ and $d$ are in minimal position, and thus, $i(d,\tau_{\gamma_2}(d)) = 2i(d,\gamma_2)^2\neq 0$. 
\end{proof}

\begin{lemma}\label{T:non commutable classes}

Let $c\subset \Sigma_g$ be a simple closed curve going through $s_i$ and $s_j$ which is not null-homotopic (as a curve in $\Sigma_g$), and $\gamma\subset \Sigma_g$ the closure of a component of $c \setminus \{s_i,s_j\}$.
For any $N\in \mathbb{Z}\setminus \{0\}$ the mapping classes $\widetilde{t_c}$ and $\tau_{\gamma}^N$ do not commute, that is, $[\widetilde{t_c},\tau_{\gamma}^N] \neq 1$ in $\Mod(\Sigma_g;\{s_1,\ldots,s_n\})$. 

\end{lemma}

\begin{proof}
Since $\widetilde{t_c}$ commutes with $\tau_{\gamma}^N$ if and only if $\widetilde{t_{\tau_{\gamma}^N(c)}}$ commutes with $\tau_{\gamma}^{-N}$, we may assume $N>0$ without loss of generality. 
Let $\gamma'$ be the closure of the complement $c\setminus \gamma$ and $d$ and $d'$ the boundaries of regular neighborhoods of $\gamma$ and $\gamma'$, respectively. 
The curve $\widetilde{t_c}(d)$ is isotopic to $d'$, while $\tau_{\gamma}^N\widetilde{t_c}\tau_{\gamma}^{-N}(d)$ is isotopic to $\tau_{\gamma}^N(d')$. 
We will prove that $d'$ and $\tau_{\gamma}^N(d')$ are not isotopic by showing $i(\tau_{\gamma}^N(d'),c) \neq i(d',c)=2$. 

The curve $\tau_{\gamma}^N(d')$ is described in Figure~\ref{F:half twisted curve}. 
\begin{figure}[htbp]
\centering
\includegraphics[height=25mm]{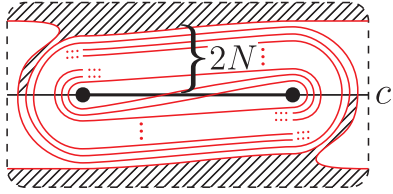}
\caption{The curve $\tau_{\gamma}^N(d')$. 
The central bold curve is $\gamma$. }
\label{F:half twisted curve}
\end{figure}
It intersects with $c$ at $4N+2$ points. 
We easily see that none of the regions made by $\tau_{\gamma}^N(d')$ and $c$ in Figure~\ref{F:half twisted curve}, except for the shaded ones, can be bigons.
Furthermore, neither of the shaded regions can be a bigon since $c$ is not null-homotopic in $\Sigma_g$.
Thus the curve $\tau_{\gamma}^N(d')$ is in minimal position with $c$, and $i(c,\tau_{\gamma}^N(d')) = 4N +2 \neq 2$ for $N \neq 0$, as claimed. 
\end{proof}

\begin{lemma}\label{T:isotopy equality lift Dehn twist}
Let $c_1, c_2\subset \Sigma_g$ be simple closed curves going through $s_i$ and $s_j$ which are not null-homotopic (as curves in $\Sigma_g$). 
Then $\widetilde{t_{c_1}}=\widetilde{t_{c_2}}$ in $\Mod(\Sigma_g;\{s_1,\ldots,s_n\})$ if and only if $c_1$ and $c_2$ are isotopic relative to the points $s_1,\ldots,s_n$. 
\end{lemma}

\begin{proof}
Once again the ``if'' part of the statement is obvious. 
Assume that $c_1$ and $c_2$ are not isotopic. 
We will prove that $\widetilde{t_{c_1}}$ and $\widetilde{t_{c_2}}$ are not equal. 
By an isotopy relative to the points $s_1,\ldots,s_n$, we change $c_1$ and $c_2$ so that these are in minimal position. 
We first note that, if we can find a simple closed curve $c_3\subset \Sigma_g\setminus\{s_1,\ldots,s_n\}$ away from $c_1$ such that $i(c_2,c_3)$ is not equal to $0$, we can prove that $\widetilde{t_{c_1}}$ and $\widetilde{t_{c_2}}$ are different mapping classes in the same way as in the proof of \cite[Fact 3.6]{Farb_Margalit_2011}. 

\vspace{.5em}

\noindent
{\bf Case 1}~:~Suppose that both of the components of $c_1\setminus \{s_i,s_j\}$ intersect $c_2$. 
We may assume that, at each of the points $s_i$ and $s_j$, either $c_1$ and $c_2$ intersect transversely or these are tangent to each other.
Let $\nu c_1$ be a tubular neighborhood of $c_1$ and $U_i, U_j$ small neighborhoods of $s_i, s_j$. 
If $c_1$ and $c_2$ are tangent to each other at both of the points, then one of the following holds: 

\begin{itemize}

\item
the intersections $U_i\cap c_2$ and $U_j\cap c_2$ are contained in the same component of $\nu c_1 \setminus c_1$ (see Figure~\ref{F:bigon pair path3}), 

\item
the component of $\nu c_1 \setminus c_1$ containing $U_i\cap c_2$ is different from that containing $U_j\cap c_2$ (see Figure~\ref{F:bigon pair path additional 1}). 

\end{itemize}

\noindent
Altogether we have to consider the four cases described in Figures~\ref{F:bigon pair path1}~--~\ref{F:bigon pair path additional 1}. 
For each case we take a parallel copy $c_1'$ of $c_1$ as shown in the figures. 
\begin{figure}[htbp]
\centering
\subfigure[]{
\includegraphics[height=37mm]{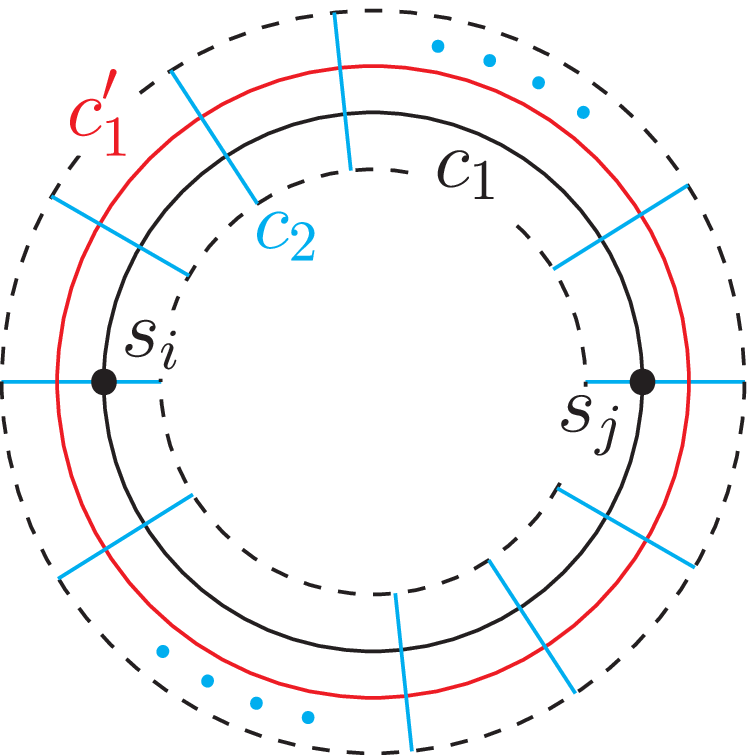}
\label{F:bigon pair path1}
}
\hspace{.3em}
\subfigure[]{
\includegraphics[height=37mm]{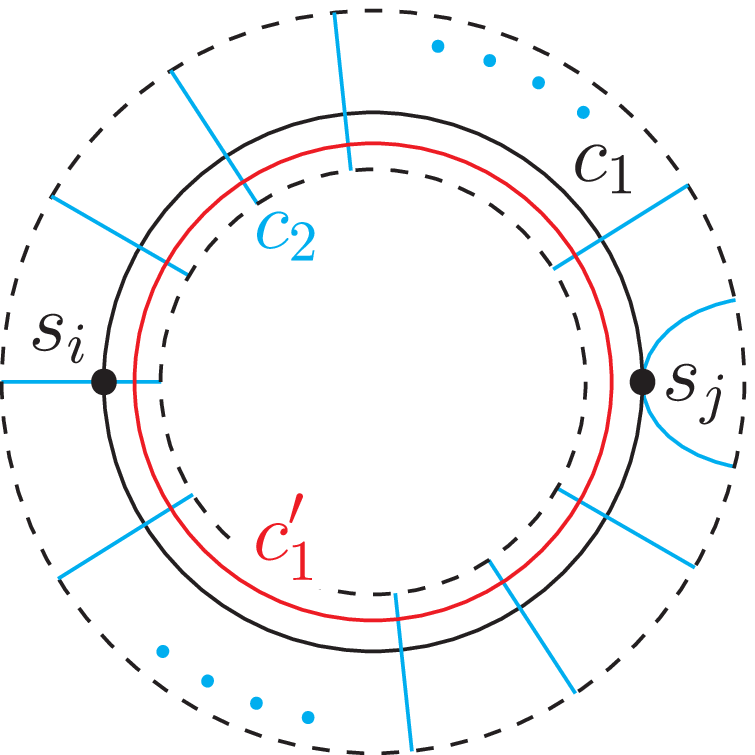}
\label{F:bigon pair path2}
}

\subfigure[]{
\includegraphics[height=37mm]{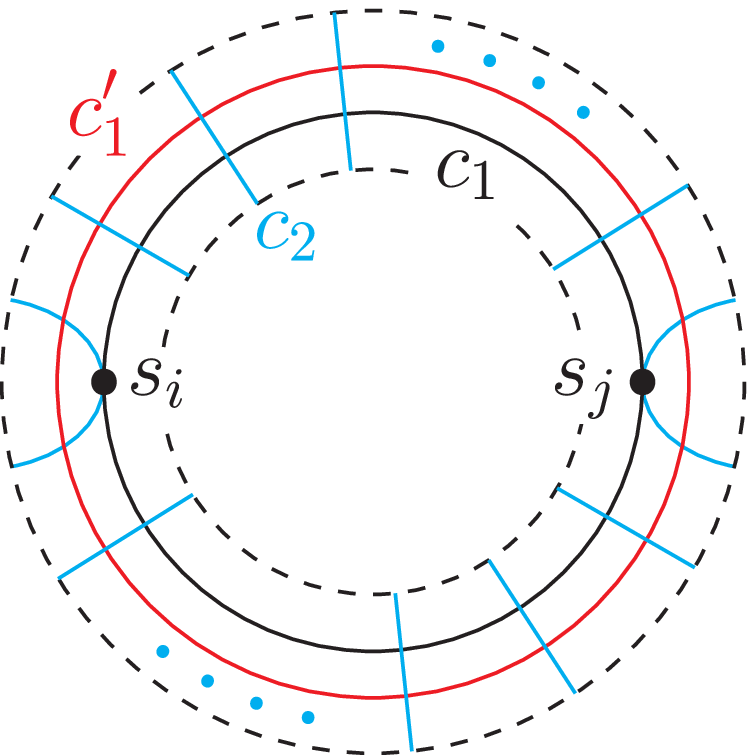}
\label{F:bigon pair path3}
}
\hspace{.3em}
\subfigure[]{
\includegraphics[height=37mm]{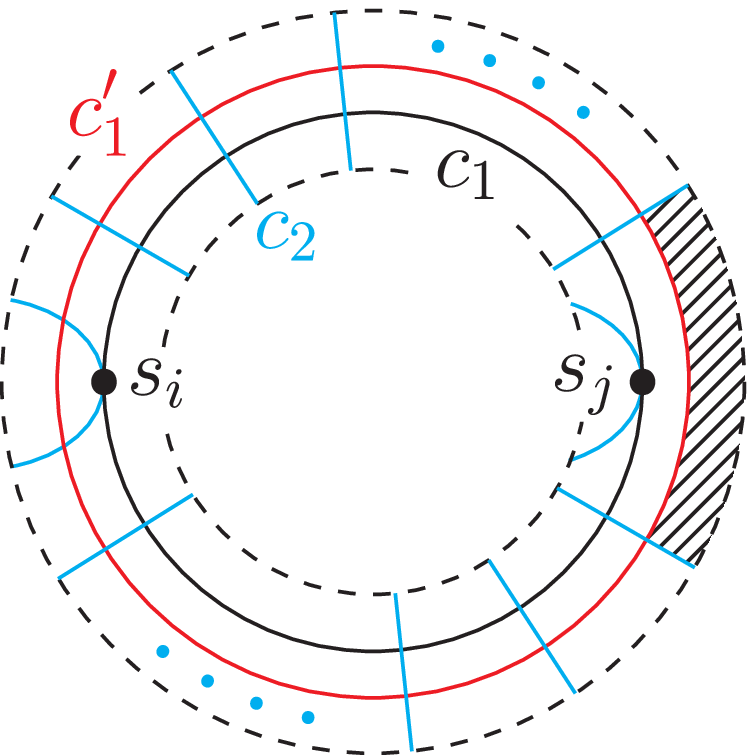}
\label{F:bigon pair path additional 1}
}
\hspace{.3em}
\subfigure[]{
	\includegraphics[height=37mm]{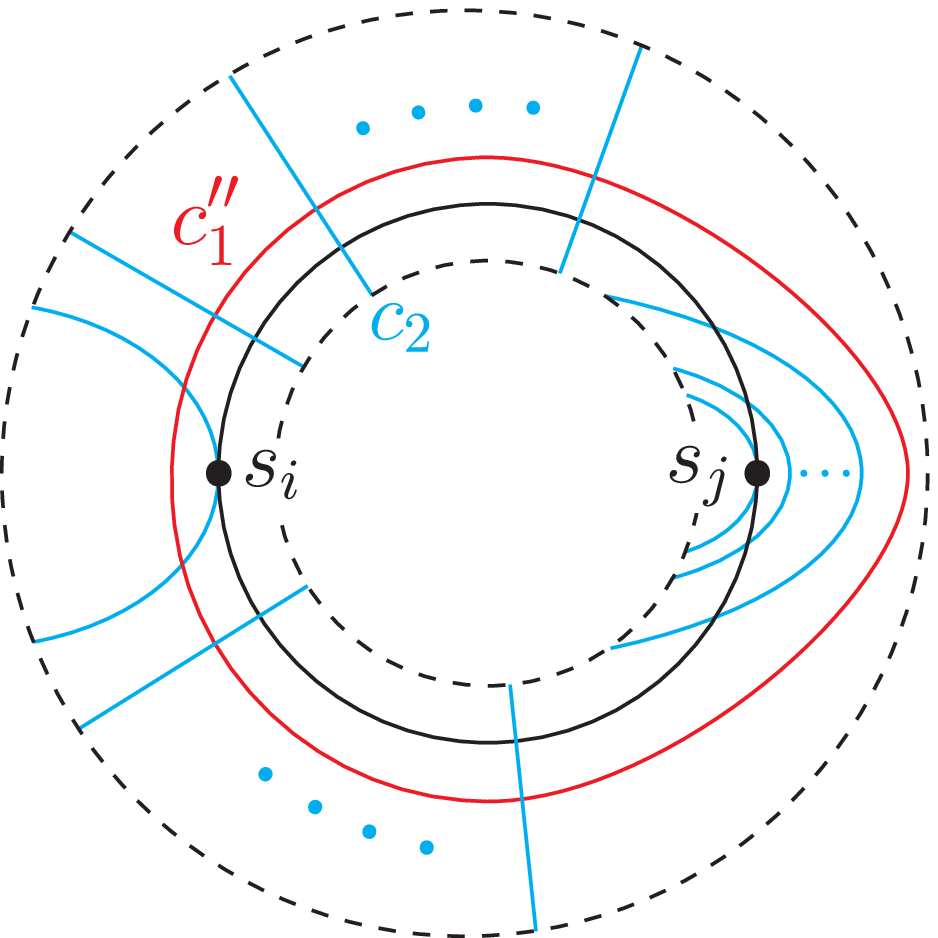}
	\label{F:bigon pair path additional 2}
}
\caption{The curves $c_1'$ and $c_1''$.}
\label{F:bigon pair paths case1}
\end{figure}
We can easily verify that no region made by $c_1'$ and $c_2$ in Figures~\ref{F:bigon pair path1}, \ref{F:bigon pair path2} and \ref{F:bigon pair path3} can be a bigon. (Recall that $c_1$ and $c_2$ are assumed to be in minimal position.)
In these cases, we put $c_3 = c_1'$, which satisfies the desired conditions (i.e. away from $c_1$ and $i(c_2,c_3) \neq 0$). 
In particular we can deduce that $\widetilde{t_{c_1}}$ and $\widetilde{t_{c_2}}$ are different mapping classes. 
As for the last case, the shaded region in Figure~\ref{F:bigon pair path additional 1} can be a bigon. 
If this region is a bigon, we move $c_1'$ to $c_1''$ so that it avoids all the bigons nested around $s_j$ (see Figure~\ref{F:bigon pair path additional 2}). 
It is easy to see that $c_1''\cap c_2$ is not empty and no region made by $c_1''$ and $c_2$ in Figure~\ref{F:bigon pair path additional 2} can be a bigon. 
Thus $c_3=c_1''$ satisfies the desired conditions. 

\vspace{.5em}

\noindent
{\bf Case 2}~:~Suppose that one of the components of $c_1\setminus \{s_i,s_j\}$ intersects $c_2$ but the other one does not. 
As in Case 1, we have to consider four cases according to the configuration of $c_2$ around $s_i$ and $s_j$. 
In each case we take parallel copies of $c_1$ as shown in Figures~\ref{F:bigon pair paths case 2}, \ref{F:bigon pair paths case 2_2} and \ref{F:bigon pair paths case 2_3}. 
\begin{figure}[htbp]
	\centering
	\subfigure[]{
		\includegraphics[height=37mm]{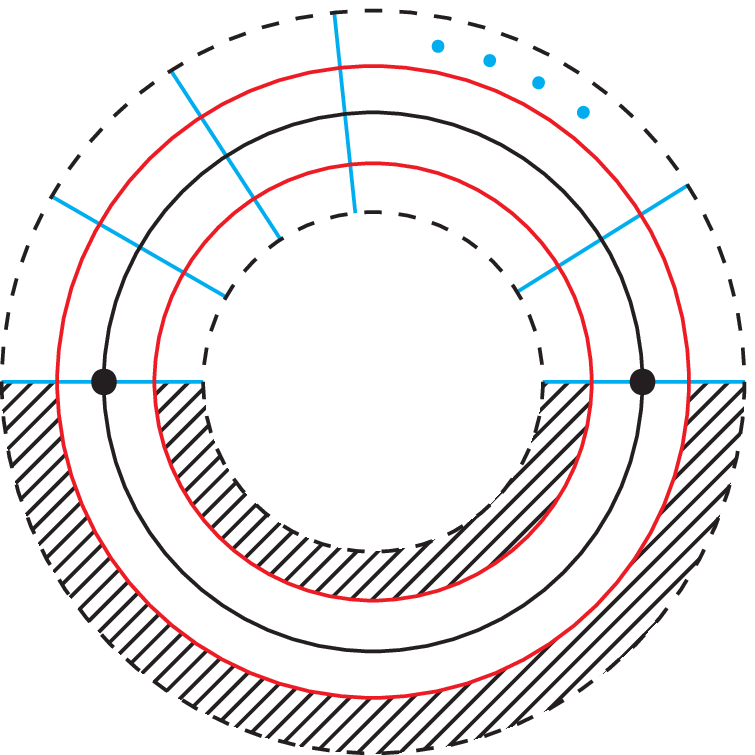}
		\label{F:bigon pair path8}
	}
	\hspace{.5em}
	\subfigure[]{
		\includegraphics[height=37mm]{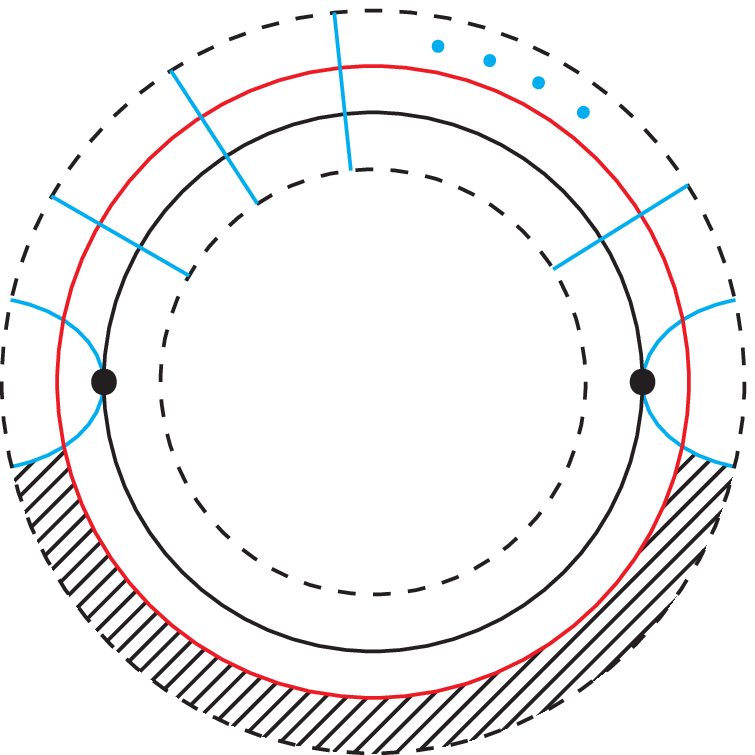}
		\label{F:bigon pair path10}
	}
	\caption{The curves parallel to $c_1$ and candidates of bigons.}
	\label{F:bigon pair paths case 2}
\end{figure}
It is easy to verify that either of the shaded regions in Figure~\ref{F:bigon pair path8} is not a bigon. 
Thus, either of the parallel curves of $c_1$ in the figure is in minimal position with $c_2$. 
If the shaded region in Figure~\ref{F:bigon pair path10} is not a bigon, then the parallel curve in the figure is in minimal position with $c_2$.
We can move $c_2$ by an isotopy so that it intersects $c_1$ transversely on both $s_i$ and $s_j$ if the shaded region in Figure~\ref{F:bigon pair path10} is a bigon. 
In this case, we can take a curve $c_3\subset \Sigma_g\setminus \{s_1,\ldots,s_n\}$ so that it is away from $c_1$ and $i(c_2,c_3) \neq 0$. 

None of the regions made by $c_2$ and the parallel copy of $c_1$ in Figure~\ref{F:bigon pair path9}, except for the shaded one, can be bigons.
\begin{figure}[htbp]
\centering
\subfigure[]{
\includegraphics[height=37mm]{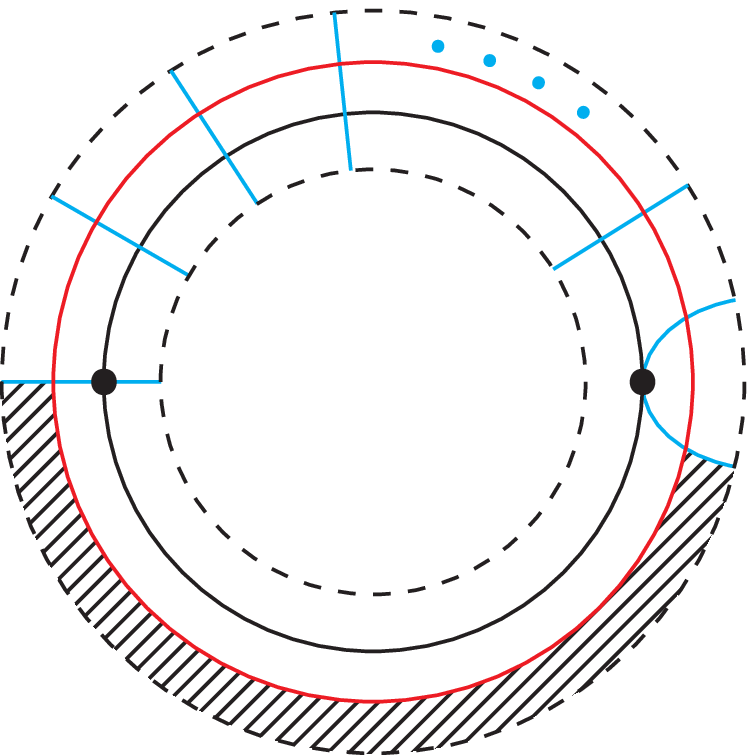}
\label{F:bigon pair path9}
}
\hspace{.5em}
\subfigure[]{
\includegraphics[height=37mm]{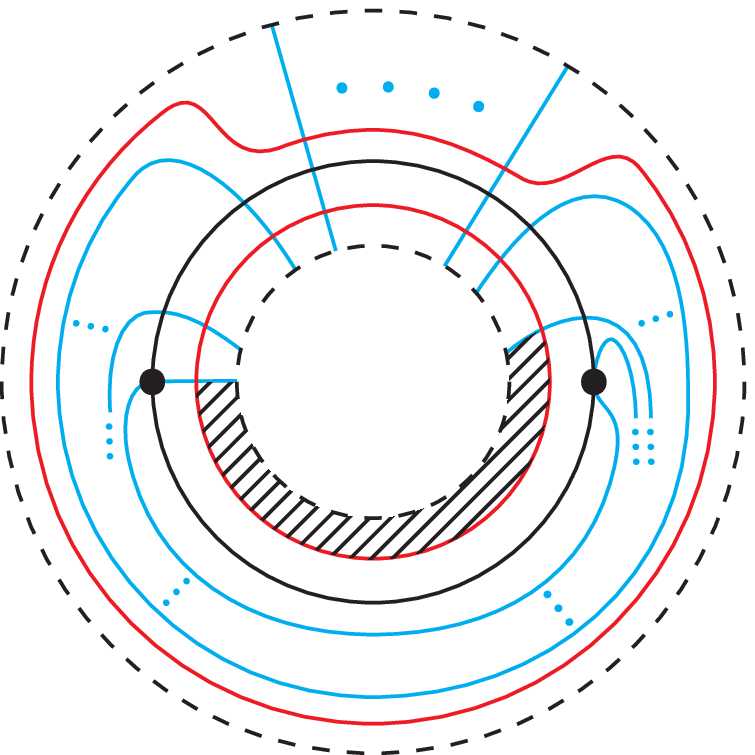}
\label{F:bigon pair path9_2}
}
\hspace{.5em}
\subfigure[]{
\includegraphics[height=37mm]{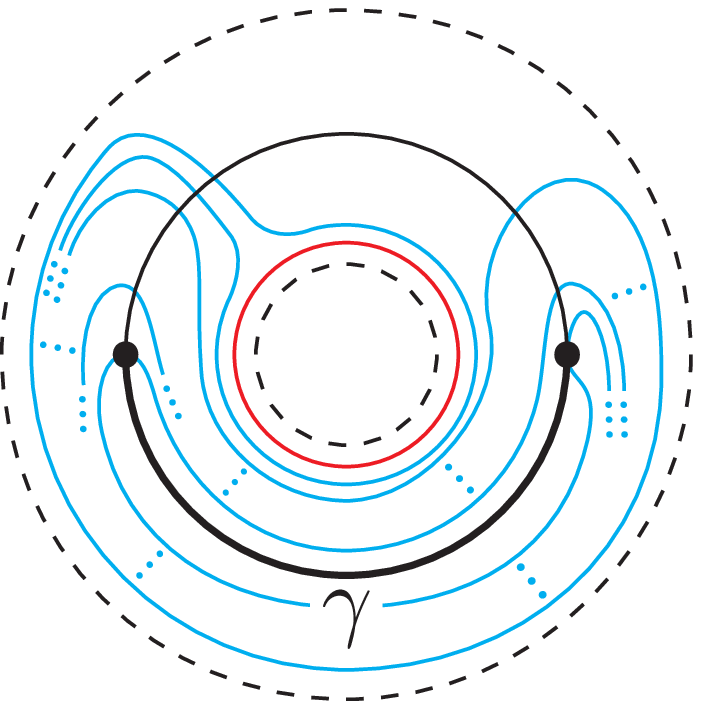}
\label{F:bigon pair path9_3}
}
	\caption{The curves parallel to $c_1$ and candidates of bigons.}
	\label{F:bigon pair paths case 2_2}
\end{figure}
If the shaded region in Figure~\ref{F:bigon pair path9} is not a bigon, then the parallel copy of $c_1$ in the figure is in minimal position with $c_2$. 
If the shaded region in Figure~\ref{F:bigon pair path9} is a bigon, we move the parallel copy of $c_1$ so that it avoids all the bigons nested around the shaded region (see Figure~\ref{F:bigon pair path9_2}). 
It is easily verified (using bigon criterion) that the resulting curve $d$ is in minimal position with $c_2$. 
If $d$ has non-empty intersection with $c_2$, we can deduce that $\widetilde{t_{c_1}}$ and $\widetilde{t_{c_2}}$ are different mapping classes. 
If $d$ is away from $c_2$, we take another parallel copy of $c_1$ ``inside'' $c_1$ as shown in Figure~\ref{F:bigon pair path9_2}. 
This copy is in minimal position provided that the shaded region in Figure~\ref{F:bigon pair path9_2} is not a bigon. 
If it is a bigon, we again move the copy so that it avoids all the bigons nested around the shaded region. 
It is easily verify that the resulting curve $d'$ is in minimal position with $c_2$. 
We can deduce $\widetilde{t_{c_1}} \neq \widetilde{t_{c_2}}$ provided that $d'$ intersects with $c_2$. 
If $d'$ is away from $c_2$, then $c_2$ is as shown in Figure~\ref{F:bigon pair path9_3}, in particular it is isotopic to $\tau_{\gamma}^N(c_1)$ for some $N >0$, where $\gamma$ is a path between $s_i$ and $s_j$ in Figure~\ref{F:bigon pair path9_3}. 
We can thus deduce from Lemma~\ref{T:non commutable classes} that $\widetilde{t_{c_1}}$ and $\widetilde{t_{c_2}} = \tau_{\gamma}^{-N}\widetilde{t_{c_1}}\tau_{\gamma}^N$ are different mapping classes.

None of the regions made by $c_2$ and the parallel copy of $c_1$ in Figure~\ref{F:bigon pair path11}, except for the shaded one, can be bigons.
\begin{figure}[htbp]
\centering
\subfigure[]{
\includegraphics[height=37mm]{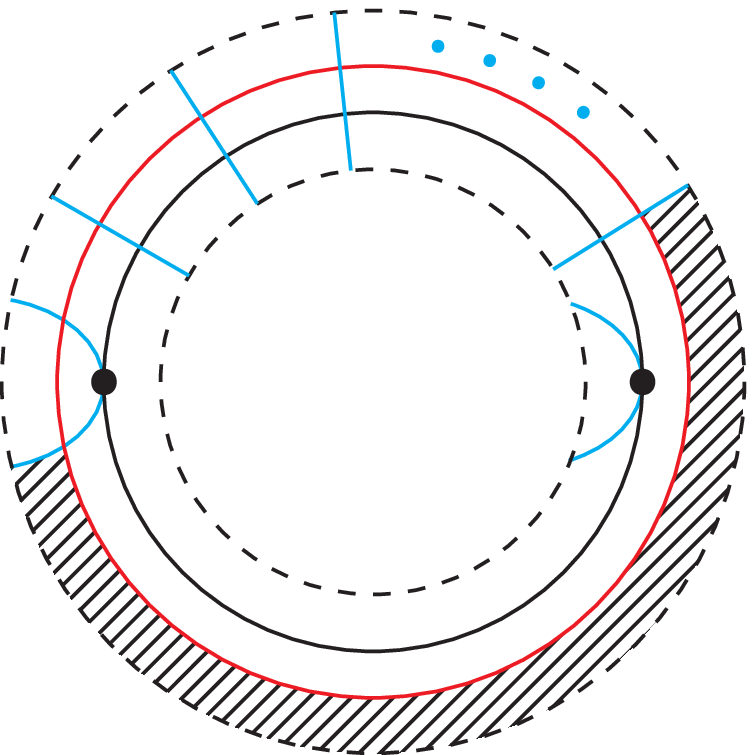}
\label{F:bigon pair path11}
}
\hspace{.5em}
\subfigure[]{
\includegraphics[height=37mm]{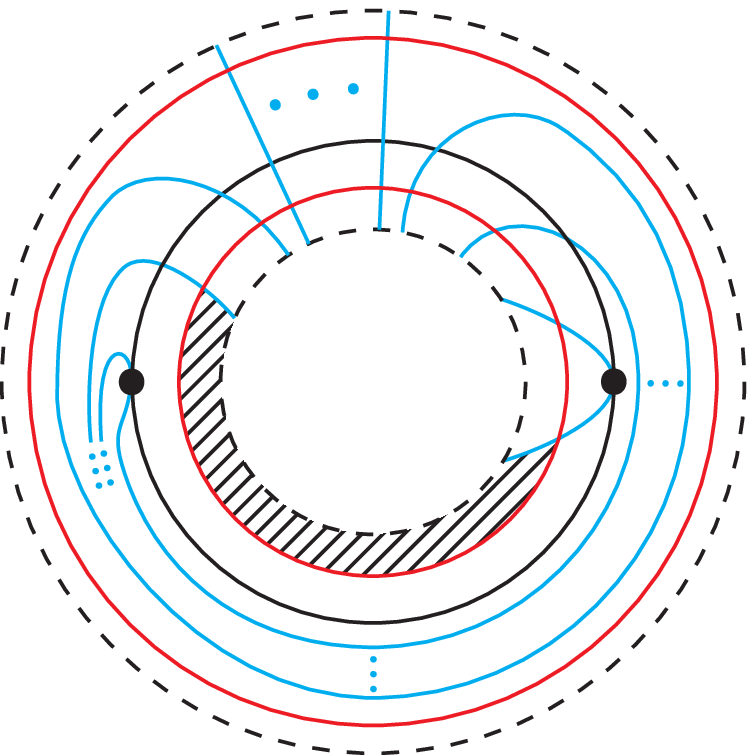}
\label{F:bigon pair path11_2}
}
\hspace{.5em}
\subfigure[]{
\includegraphics[height=37mm]{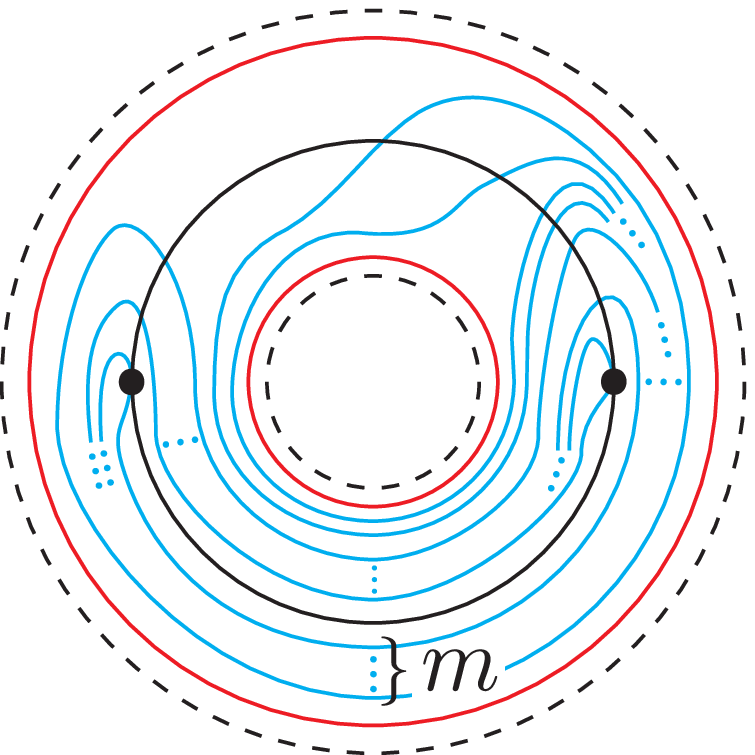}
\label{F:bigon pair path11_3}
}
	\caption{The curves parallel to $c_1$ and candidates of bigons.}
	\label{F:bigon pair paths case 2_3}
\end{figure}
Thus we can deduce $\widetilde{t_{c_1}} \neq \widetilde{t_{c_2}}$ provided that the shaded region in Figure~\ref{F:bigon pair path11} is not a bigon. 
If the shaded region is a bigon, we move the parallel copy by an isotopy so that it avoids all the bigons nested around the shaded one in Figure~\ref{F:bigon pair path11} (see Figure~\ref{F:bigon pair path11_2}). 
It is easily verified that the resulting curve $\tilde{d}$ is in minimal position with $c_2$, so we can deduce $\widetilde{t_{c_1}} \neq \widetilde{t_{c_2}}$ provided that $\tilde{d}$ intersects with $c_2$. 
If $\tilde{d}$ is away from $c_2$, we take another parallel copy of $c_1$ ``inside'' $c_1$ as shown in Figure~\ref{F:bigon pair path11_2}.
This copy is in minimal position provided that the shaded region in Figure~\ref{F:bigon pair path11_2} is not a bigon. 
If it is a bigon, we again move the copy so that it avoids all the bigons nested around the shaded region. 
It is easy to check that the resulting curve $\tilde{d}'$ is in minimal position with $c_2$. 
The curve $\tilde{d}'$ must intersect with $c_2$. 
For, if $\tilde{d}'$ were away from $c_2$, $\tilde{d}'$ is as shown in Figure~\ref{F:bigon pair path9_3}, in particular it would be disconnected if the number $m$ of strands of paths is odd, or null-homotopic (as a curve in $\Sigma_g$) if $m$ is even, but both of the consequences contradict our initial assumptions.
We conclude that $\widetilde{t_{c_1}}$ and $\widetilde{t_{c_2}}$ are different mapping classes. 

\vspace{.5em}

\noindent
{\bf Case 3}~:~Suppose that neither of the components of $c_1\setminus \{s_i,s_j\}$ intersect $c_2$. 
As before, we consider four cases according to the configuration of $c_2$ around $s_i$ and $s_j$. 
In each case we take parallel copies of $c_1$ as shown in Figure~\ref{F:bigon pair paths case 3}. 
\begin{figure}[htbp]
\centering
\subfigure[]{
\includegraphics[height=37mm]{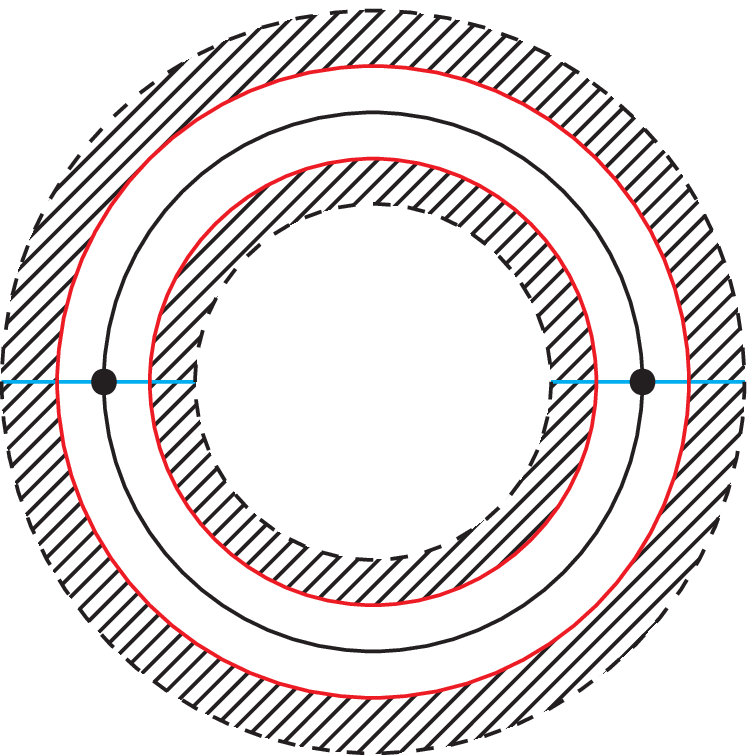}
\label{F:bigon pair path4}
}
\hspace{.5em}
\subfigure[]{
\includegraphics[height=37mm]{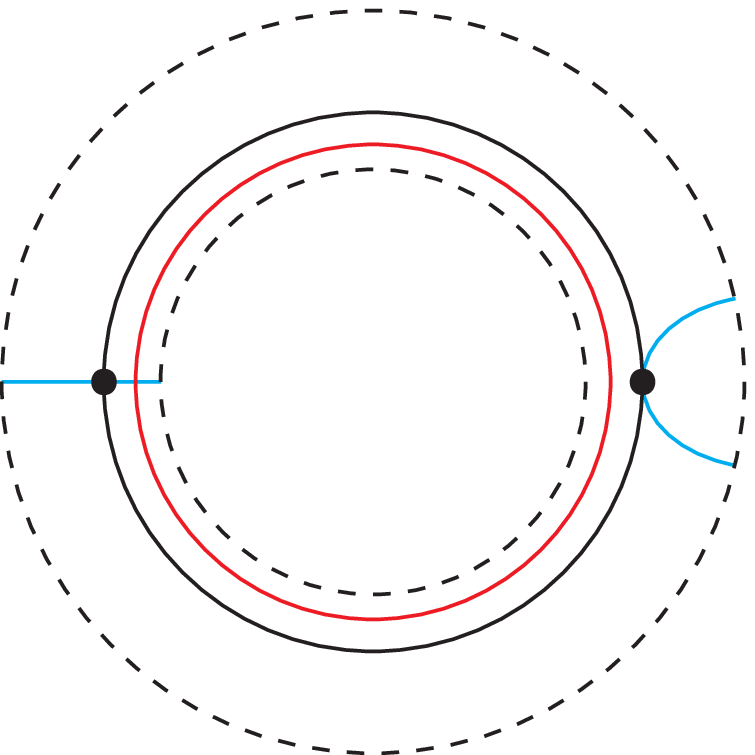}
\label{F:bigon pair path5}
}

\subfigure[]{
\includegraphics[height=37mm]{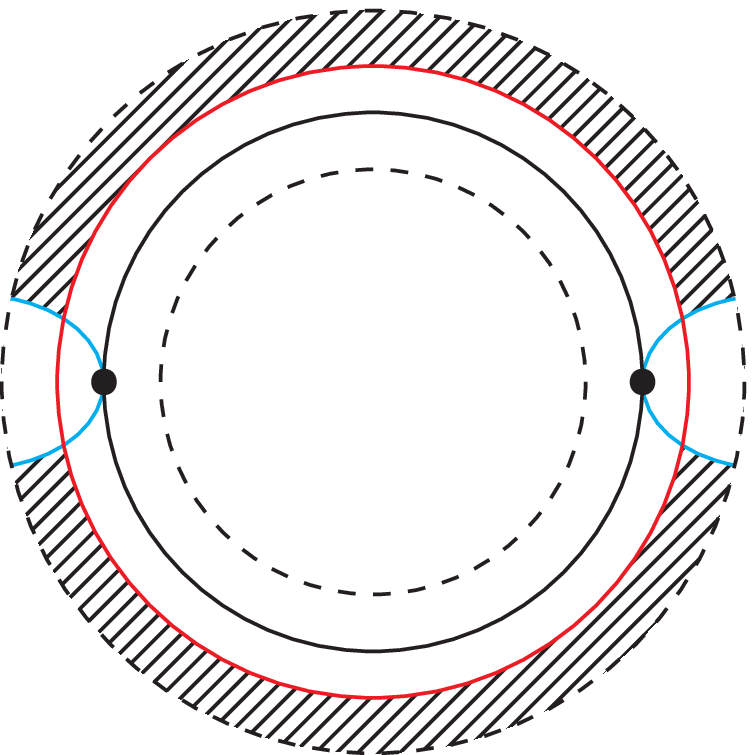}
\label{F:bigon pair path6}
}
\subfigure[]{
	\includegraphics[height=37mm]{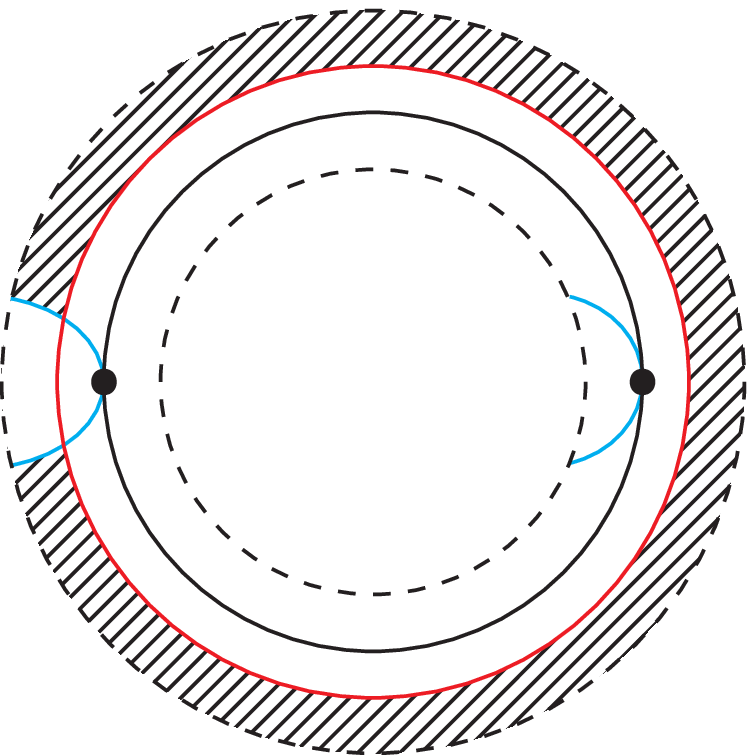}
	\label{F:bigon pair path7}
}
\caption{The curve $c_1'$ and candidates of bigons.}
\label{F:bigon pair paths case 3}
\end{figure}
If both of the parallel copies of $c_1$ in Figure~\ref{F:bigon pair path4} formed bigons with $c_2$, either $c_2$ would be homotopic to $c_1$ or $c_2$ is null-homotopic (as a simple closed curve in $\Sigma_g$), which contradicts the assumptions. 
Thus either one of the copies in Figure~\ref{F:bigon pair path4} is in minimal position with $c_2$, and we can deduce $\widetilde{t_{c_1}} \neq \widetilde{t_{c_2}}$. 
The parallel copy of $c_1$ in Figure~\ref{F:bigon pair path5} is in minimal position with $c_2$ since these intersect at a single point. 
If both of the shaded regions in Figure~\ref{F:bigon pair path6} are not bigons, $c_2$ and the copy in the figure are in minimal position.  
If either of the shaded regions in Figure~\ref{F:bigon pair path6} is a bigon, we can move $c_1$ by a isotopy so that $c_1$ and $c_2$ intersect transversely at $s_i$ and $s_j$. 
In both cases, we can deduce $\widetilde{t_{c_1}} \neq \widetilde{t_{c_2}}$. 
If the shaded region in Figure~\ref{F:bigon pair path7} were a bigon, $c_2$ would be homotopic to $c_1$, which contradicts the assumptions. 
Thus, $c_2$ and the copy in Figure~\ref{F:bigon pair path7} are in minimal position.  
\end{proof}

\begin{remark}
We should point out that one cannot state Lemmas~\ref{T:isotopy equality half twist} and \ref{T:isotopy equality lift Dehn twist} for the lifts of the elements featured in them to the framed mapping class group $\Mod(\Sigma_g^n;\{u_1,\ldots,u_n\})$ instead. For example, for the paths $\gamma$ and $\gamma' = t_{\delta_i}t_{\delta_j}(\gamma)$ between $s_i$ and $s_j$, which are not isotopic (relative to the boundary of $\Sigma_g^n$) in general, we can see that $\tau_{\gamma}$ is equal to $\tau_{\gamma'}$ in $\Mod(\Sigma_g^n;\{u_1,\ldots,u_n\})$. There are in fact infinitely many such lifts of arc twists in $\Mod(\Sigma_g;\{s_1,\ldots,s_n\})$, which is the underlying cause for this ambiguity.
\end{remark}

Next is a variation of a classical result of Earle and Schatz \cite{Earle_Schatz}:

\begin{lemma}\label{T:contractibility Diff}
If $2-2g-n$ is negative, then $\pi_1(\Diff(\Sigma_g;\{s_1,\ldots,s_n\}),\id)$ is trivial. 
\end{lemma}

\begin{proof}
As in the proof of \cite[Theorem 4.6]{Farb_Margalit_2011}, we can obtain the following exact sequence (note that we omit the base points for simplicity): 
\begin{equation}\label{E:exact sequence configuration space}
\begin{split}
\pi_2(F_{0,n}(\Sigma_g))\to \pi_1(\Diff(\Sigma_g;\{s_1,\ldots,s_n\})) \to \pi_1(\Diff(\Sigma_g)) \\
\to \pi_1(F_{0,n}(\Sigma_g)), 
\end{split}
\end{equation}
where $F_{0,n}(\Sigma_g)$ is the configuration space defined in \cite{Fadell_Neuvirth}, which is aspherical if $g\geq 1$. 
Since the group $\pi_1(\Diff(\Sigma_g))$ is trivial for $g\geq 2$ (\cite{Earle_Eells}), so is $\pi_1(\Diff(\Sigma_g;\{s_1,\ldots,s_n\}))$ if $g\geq 2$.  
Furthermore, the following diagram commutes: 
\[
\xymatrix{
& \pi_1(T^2) \ar[dl]_{\cong} \\
\pi_1(\Diff(T^2)) \ar[r]^{\mathrm{ev}_\ast} & \pi_1(F_{0,n}(T^2)) \ar[u]_{\pi_\ast} , 
}
\]
where $\pi_\ast$ is induced by the natural projection and $\mathrm{ev}_\ast$ is induced by the evaluation map, which is the same map as that in \eqref{E:exact sequence configuration space}. 
In particular, $\mathrm{ev}_\ast$ is injective. 
We can thus deduce from the exact sequence \eqref{E:exact sequence configuration space} that $\pi_1(\Diff(T^2;\{s_1,\ldots,s_n\}))$ is trivial for $n>0$. 

We can also obtain the following exact sequence: 
\begin{equation}\label{E:exact sequence configuration sphere}
\begin{split}
\pi_2(F_{n-1,1}(S^2)) \to \pi_1(\Diff(S^2;\{s_1,\ldots,s_n\})) \to \pi_1(\Diff(S^2;\{s_1,\ldots,s_{n-1}\}))\\
\to \pi_1(F_{n-1,1}(S^2)). 
\end{split}
\end{equation}
The configuration space $F_{n-1,1}(S^2)$ is aspherical for $n\geq 2$ (\cite{Fadell_Neuvirth}). 
Thus, if the fundamental group $\pi_1(\Diff(\Sigma_g;\{s_1,\ldots,s_{n-1}\}))$ is trivial, so is $\pi_1(\Diff(\Sigma_g;\{s_1,\ldots,s_n\}))$. 
Using \eqref{E:exact sequence configuration sphere} we can verify that $\pi_1(\Diff(S^2;\{s_1\}))$ is an infinite cyclic group generated by the loop $\theta \to \phi_{2\pi \theta}\in \Diff(S^2;\{s_1\})$, where $\phi_{\theta}$ is the $\theta$--degree rotation of $S^2$ fixing $s_1$. 
Since $F_{1,1}(S^2)= \R^2$, especially $\pi_1(F_{1,1}(S^2))=1$, $\pi_1(\Diff(\Sigma_g;\{s_1,s_2\}))$ is also an infinite cyclic group generated by $\phi_{2\pi \theta}$. 
It is easy to see that $[\phi_{2\pi \theta}]\in \pi_1(\Diff(\Sigma_g;\{s_1,s_2\}))$ is sent to the generator of $\pi_1(F_{2,1}(S^2))\cong \Z$. 
We can eventually conclude that $\pi_1(\Diff(\Sigma_g;\{s_1,\ldots,s_n\}))$ is trivial for any $n\geq 3$. 
\end{proof}

\begin{remark}
The above lemma can be possibly derived as a corollary of the contractibility of the identity component of $\Diff(\Sigma_g^n)$, as shown in \cite{Earle_Schatz},
provided $\Diff(\Sigma_g^n)$ is seen to be homotopy equivalent to $\Diff(\Sigma_g;\{s_1,\ldots,s_n\})$. For our purposes however, it is sufficient to calculate the fundamental group of $\Diff(\Sigma_g;\{s_1,\ldots,s_n\})$. So we have given a direct proof of Lemma~\ref{T:contractibility Diff}.  
\end{remark}

We are now ready to prove the theorem.
\begin{proof}[Proof of Theorem~\ref{T:equivalence multisection factorization}]
We first assume that $(X_1,f_1,S_1)$ and $(X_2,f_2,S_2)$ are equivalent. 
We can take diffeomorphisms $\Phi:X_1\to X_2$ and $\phi:S^2\to S^2$ such that $\phi\circ f_1 = f_2\circ \Phi$ and $\Phi(S_1)=S_2$. 
Let $\alpha_1,\ldots, \alpha_{k+l}\subset S^2$ be reference paths for $f_1$ with the common initial point $p_0$ which give rise to the factorization $W_{X_1,f_1,S_1}$ under an identification $\Theta:(f_1^{-1}(p_0),f_1^{-1}(p_0)\cap S_1)\xrightarrow{\cong} (\Sigma_g,\{s_1,\ldots, s_n\})$.  
Let $p_0' = \phi(p_0)$, $\alpha_i' = \phi(\alpha_i)$ and $\Theta' = \Theta\circ \Phi^{-1}$, which is an identification of $(f_2^{-1}(p_0'),f_2^{-1}(p_0')\cap S_2)$ with $(\Sigma_g,\{s_1,\ldots, s_n\})$. 
It is easily verify that the monodromy factorization of $f_2$ obtained from $\alpha_1',\ldots,\alpha_{k+l}'$ and $\Theta'$ coincides with $W_{X_1,f_1,S_1}$. 
Since any two factorizations of $(X_2,f_2,S_2)$ are Hurwitz equivalent, so are $W_{X_1,f_1,S_1}$ and $W_{X_2,f_2,S_2}$. 

In what follows we assume that $W_{X_1,f_1,S_1}$ and $W_{X_2,f_2,S_2}$ are Hurwitz equivalent. 
We first consider the case that $f_i$ has no critical points and $f_i|_{S_i}$ has no branched points. 
In this case, $f_i$ can be obtained by pasting two trivial surface bundles over the disk so that the marked points corresponding $n$--sections match. 
By Lemma~\ref{T:contractibility Diff} such a pasting map is unique up to isotopy preserving fibration structures. 
Thus $f_1$ and $f_2$ are equivalent. 

Assume that $f_i$ has critical points or $f_i|_{S_i}$ has branched points. 
We can take reference paths $\alpha^i_1,\ldots,\alpha^i_{k+l}\subset S^2$ for $f_i$ so that the local monodromies associated with $\alpha^1_j$ and $\alpha^2_j$ coincide. 

By composing a self-diffeomorphism of $S^2$ to $f_2$, we may assume that $f_1(\Crit(f_1)\cup \Crit(f_1|_{S_1}))$ and $f_2(\Crit(f_2)\cup \Crit(f_2|_{S_2}))$ coincide. 
Let $f_1(\Crit(f_1)\cup \Crit(f_1|_{S_1}))=\{a_1,\ldots,a_{k+l}\}$ and $D_j\subset S^2$ a sufficiently small disk neighborhood of $a_j$.
Since all the local monodromies of $f_1$ and $f_2$ coincide, we can take a diffeomorphism $H:f_1^{-1}(S^2\setminus \amalg_jD_j)\to f_2^{-1}(S^2\setminus \amalg_jD_j)$ sending the intersection $f_1^{-1}(S^2\setminus \amalg_jD_j)\cap S_1$ to $f_2^{-1}(S^2\setminus \amalg_jD_j)\cap S_2$ such that the following diagram commutes: 
\[
\xymatrix{
f_1^{-1}(S^2\setminus \amalg_jD_j) \ar[r]^{H} \ar[d]^{f_1} & f_2^{-1}(S^2\setminus \amalg_jD_j) \ar[dl]^{f_2} \\
S^2\setminus \amalg_jD_j. 
}
\]
In what follows we will extend $H$ to a diffeomorphism with the source containing the preimages $f_1^{-1}(D_1), \ldots, f_1^{-1}(D_{k+l})$. 

If $f_i^{-1}(a_j)$ contains a Lefschetz critical point which is not a branched point of $f_i|_{S_i}$, then in the same manner as in the proof of \cite[Theorem 2.4]{Matsumoto3} we can extend $H$ to a diffeomorphism with the source containing $f_1^{-1}(D_j)$ (in this procedure we need Lemma~\ref{T:contractibility Diff} instead of the contractibility of $\Diff(\Sigma_g)$ used in \cite[p.133]{Matsumoto3}). 

Assume that $f_i^{-1}(a_j)$ contains a Lefschetz critical point $x_{ij}\in X_i$ which is also a branched point of $f_i|_{S_i}$. 
There exist complex coordinate neighborhoods $(U^i,\varphi^i)$ and $(V^i,\psi^i)$ at $x_{ij} \in X_i$ and $f_i(x_{ij})\in S^2$, respectively, which satisfy the following properties: 
\begin{enumerate}

\item
$\psi^i\circ f_i\circ (\varphi^i)^{-1}(z,w) = z^2 + w^2$, 

\item
$\varphi^i(U^i\cap S_i) = \C \times \{0\}$. 

\end{enumerate}
Using the disk theorem as in the proof of \cite[Theorem 2.4]{Matsumoto3}, we may assume that $\psi^1$ and $\psi^2$ coincide without loss of generality. 
By Lemma~\ref{T:isotopy equality lift Dehn twist} the vanishing cycles associated with $\alpha^1_j$ and $\alpha^2_j$ coincide up to isotopy relative to the points $s_1,\ldots,s_n$. 
Thus, in the same way as that in the proof of \cite[Lemma 2.5]{Matsumoto3}, we can change $H$ by a vertical isotopy (in the sense of \cite{Matsumoto3}) sending $S_1$ to $S_2$ at all times so that $\varphi^2\circ H = \varphi^1$ on a neighborhood (in $f_1^{-1}(\Pa D_j)$) of the vanishing cycle of $f_1$ associated with $\alpha^1_j$. 
The arguments following the proof of \cite[Lemma 2.5]{Matsumoto3} can be applied to our situation, and we can eventually extend $H$ to a diffeomorphism with the source containing $f_1^{-1}(D_j)$.  

Lastly, if $f_i^{-1}(a_j)$ contains a branched point of $f_i|_{S_i}$ away from $\Crit(f_i)$, then we can extend $H$ to a diffeomorphism with the source containing $f_1^{-1}(D_j)$ in a manner quite similar to that in the previous paragraph, where we invoke Lemma~\ref{T:isotopy equality half twist} instead of Lemma~\ref{T:isotopy equality lift Dehn twist} this time. 
\end{proof}

\smallskip
\subsection{Monodromy factorizations in the framed mapping class group} \

As discussed in the previous section (cf. Remark~\ref{R:lift factorization}) we can take a lift of the factorization $W_{X,f,S}$ to that of a product of Dehn twists along boundary components in the framed mapping class group $\Mod(\Sigma_g^n;\{u_1,\ldots,u_n\})$. Such a lift is needed to fully capture the local topology of the multisection $S$. 

Two such lifts $\widetilde{W}_{X_i,f_i,S_i}$ of $W_{X_i,f_i,S_i}$, for $i=1,2$,  are not necessarily related by elementary transformations and simultaneous conjugations even if $(X_1,f_1,S_1)$ and $(X_2,f_2,S_2)$ are equivalent. There is indeed no canonical way to choose lifts of $\tau_{\gamma}$ and $\widetilde{t_c}$. For instance, each one of the paths $\gamma_i$, $i=1,2$ in Figure~\ref{F:lifts path} is a lift of $\gamma$, and in turn, $\tau_{\gamma_i}\in \Mod(\Sigma_g^n;\{u_1,\ldots,u_n\})$ is a lift of $\tau_{\gamma}\in \Mod(\Sigma_g;\{s_1,\ldots,s_n\})$.

\begin{figure}[htbp]
\centering
\includegraphics[width=70mm]{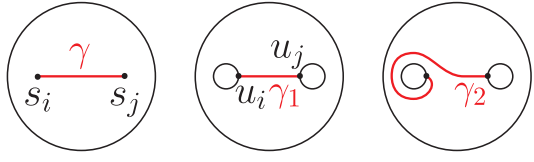}
\caption{Examples of paths between $s_i$ and $s_j$.}
\label{F:lifts path}
\end{figure}
\noindent It is easy to see that $\tau_{\gamma_2}$ is equal to $t_{\delta_i}\tau_{\gamma_1}t_{\delta_i}^{-1}=\tau_{\gamma_1}t_{\delta_j}t_{\delta_i}^{-1}$, where $\delta_k$ is a simple closed curve along the boundary component containing $u_k$. 
Thus, different choices of lifts of $\tau_{\gamma}$ and $\widetilde{t_c}$ in $W_{X_i,f_i,S_i}$ might yield a factorization $\widetilde{W}_{X_i,f_i,S_i}$ with distinct products, in particular the Hurwitz equivalence class of $\widetilde{W}_{X_i,f_i,S_i}$ depends on the choices of lifts of members in $W_{X_i,f_i,S_i}$.

For lifts of $\tau_{\gamma}$ and $\widetilde{t_c}$ are uniquely determined up to conjugations by Dehn twists along boundary components however, $\widetilde{W}_{X_i,f_i,S_i}$ is uniquely determined up to elementary transformations, simultaneous conjugations, plus a third modification:

\begin{enumerate}

\setcounter{enumi}{2}

\item
\emph{\Thirdmove}, which changes a factorization as follows: 
\[
\xi_{k+l}\cdots \xi_{i+1}\xi_i\xi_{i-1}\cdots \xi_1 \longleftrightarrow \xi_{k+l}\cdots \xi_{i+1}(t_{\delta}\xi_it_{\delta}^{-1})\xi_{i-1}\cdots \xi_1, 
\]
where $\delta$ is a simple closed curve along a boundary component of $\Sigma_g^n$. 

\end{enumerate}

\noindent 
A {\thirdmove} does not affect usual Dehn twists among the $\xi_i$ factors since any simple closed curve can be isotoped away from the boundary (but does affect lifts of Dehn twists). 
In particular, this move is not needed  to relate monodromy factorizations in $\Mod(\Sigma_g^n;\{u_1,\ldots,u_n\})$ associated to \emph{pure multisections}, i.e. disjoint union of sections. 

We say that two positive factorizations of products of Dehn twists along boundary components in $\Mod(\Sigma_g^n;\{u_1,\ldots,u_n\})$ are \emph{Hurwitz equivalent}, if one can be obtained from the other by a sequence of elementary transformations, simultaneous conjugations and \thirdmoves.  

For $i=1,2$, let $(X_i,f_i,S_i)$ be a genus--$g$ Lefschetz fibration with an $n$--section $S_i$, with monodromy factorization $W_{X_i, f_i, S_i}$ in $\Mod(\Sigma_g;\{s_1,\ldots,s_n\})$. Let $\widetilde{W}_{X_i,f_i,S_i}$ be a lift of $W_{X_i, f_i, S_i}$, a positive factorization of the form
\[
\widetilde{W}_{X_i,f_i,S_i}:\widetilde{\xi}_{k+l}\cdots \widetilde{\xi}_{1} = t_{\delta_1}^{a_1}\cdots t_{\delta_n}^{a_n}
\]
in $\Mod(\Sigma_g^n;\{u_1 \ldots,u_n\})$. Then
Theorem~\ref{T:equivalence multisection factorization}, together with the above observation, thus yields to a one-to-one correspondence in this setting as well:

\begin{corollary}\label{T:equivalence multisection factorization framed}
Suppose that $2-2g-n$ is negative. The triples $(X_1,f_1,S_1)$ and $(X_2,f_2,S_2)$ are equivalent if and only if $\widetilde{W}_{X_1,f_1,S_1}$ and $\widetilde{W}_{X_2,f_2,S_2}$ are Hurwitz equivalent. 
\end{corollary}

\smallskip

\begin{remark}
It is worth noting that a {\thirdmove} can also affect the right hand side of a positive factorization
\[ \xi_{k+l}\cdots \xi_1=t_{\delta_1}^{a_1}\cdots t_{\delta_n}^{a_n} \]
in the framed mapping class group, simultaneously increasing and decreasing the powers of the involved boundary twists. In order to make the effect of a {\thirdmove} clear, let us define the following surjective homomorphism
\[
\Lambda:\Mod(\Sigma_g^n;\{u_1,\ldots,u_n\}) \twoheadrightarrow \mathfrak{S}_n
\]
defined by the action of a mapping class on the set $\{u_1,\ldots,u_n\}$, where $\mathfrak{S}_n$ is the symmetric group of order $n$. A Dehn twist $t_c$ is contained in the kernel of $\Lambda$, while a half twist $\tau_{\gamma}$ and a lift $\widetilde{t_c}$ of a Dehn twist are sent to transpositions by $\Lambda$. Clearly  $\xi \, t_{\delta_i}$ is equal to $t_{\delta_{\Lambda(\xi)(i)}}\xi$ for any $\xi\in \Mod(\Sigma_g^n;\{u_1,\ldots,u_n\})$, so a {\thirdmove} changes the right hand side of the factorization as follows:
\[
\xi_{k+l}\cdots \xi_{i+1}(t_{\delta_j}\xi_i t_{\delta_j}^{-1})\xi_{i-1}\cdots  \xi_1 = t_{\delta_{\Lambda(\xi_{k+l}\cdots\xi_{i+1})(j)}}t_{\delta_{\Lambda(\xi_{k+l}\cdots\xi_{i})(j)}}^{-1} t_{\delta_1}^{a_1}\cdots t_{\delta_n}^{a_n}. 
\] 
\end{remark}

\smallskip
\subsection{Equivalence of Lefschetz pencils} \

Given a genus--$g$ Lefschetz pencil $(X,f)$ with base locus $B=\{x_1, \ldots, x_n \}$, recall that we can pass to a genus--$g$ Lefschetz \emph{fibration} $(X',f')$, with a distinguished \emph{pure} $n$--section $S$ that consists of $n$ disjoint sections $S_j$ of self-intersection $-1$, each arising as an exceptional sphere in $X'=X \# n  \CPb$ of the blow-up at the base point $x_j \in X$. Since one can blow-down all $S_j$ to recover the pencil $(X,f)$, we can work with the well-known monodromy factorization $\widetilde{W}_{X',f',S}$ of $(X',f',S)$ of the form
\[
t_{c_{l}}\cdots t_{c_{1}} = t_{\delta_1} \cdots t_{\delta_n}
\]
in the framed mapping class group $\Mod(\Sigma_g^n;\{u_1,\ldots,u_n\})$. Note that this factorization is in fact contained in the subgroup $\Mod_{\partial \Sigma_g^n} (\Sigma_g^n)$ (whose elements restrict to identity along $\partial \Sigma_g^n$)  which only captures \textit{pure} $n$--sections, but we need the larger group in order to factor in pencil automorphisms which swap base points, which we will discuss shortly. This associated factorization is what we will call  monodromy factorization of the \emph{pencil} $(X,f)$, and with the above correspondence in mind, we will denote it simply by $W_{X,f}$.

Lefschetz pencils $f_i:X_i\setminus B_i \to S^2$, $i=1,2$, are said to be \emph{equivalent} if there exist orientation-preserving diffeomorphism $\Phi:X_1\to X_2$ and $\phi:S^2\to S^2$ such that $\Phi(B_1) = B_2$ and $\phi\circ f_1 = f_2 \circ \Phi$.  Clearly, for $(X_i, f_i)$ to be equivalent pencils, they should both have the same fiber genus and the same number of base points $|B_1|=|B_2|$. It now follows from Theorem~\ref{T:equivalence multisection factorization} that:

\begin{corollary}\label{T:equivalence pencils}
Two Lefschetz pencils $(X_i, f_i)$ of genus $g \geq 1$ with $n$ base points are equivalent if and only they have Hurwitz equivalent monodromy factorizations $W_{X_i,f_i}$ in $\Mod(\Sigma_g^n;\{u_1,\ldots,u_n\})$.
\end{corollary}

\begin{proof}
Since $ g\geq 1$ and $n >0$, we have $2-2g- n <0$. If $f_1$ and $f_2$ are equivalent, then the corresponding pairs of Lefschetz fibrations and these sections are also equivalent. Thus, $W_{X_1,f_1}$ and $W_{X_2,f_2}$ are Hurwitz equivalent by Corollary~\ref{T:equivalence multisection factorization framed}.  Suppose that $W_{X_1,f_1}$ and $W_{X_2,f_2}$ are Hurwitz equivalent.  Let $f'_i:X'_i=X_i\# n \CPb\to S^2$ be the associated Lefschetz fibration. We deduce from Theorem~\ref{T:equivalence multisection factorization}, and the assumption that there exist diffeomorphisms $\Phi':X'_1 \to X'_2$ and $\phi:S^2\to S^2$ such that $\Phi'$ sends the union of exceptional spheres of $X'_1$ to that of $X'_2$ and $\phi\circ {f'_1}={f'_2}\circ {\Phi'}$. So $\Phi'$ induces a diffeomorphism $\Phi:X_1\to X_2$ which satisfies $\Phi(B_1)=B_2$ and $\phi\circ f_1 = f_2\circ \Phi$, providing an equivalence between $(X_1,f_1)$ and $(X_2,f_2)$.
\end{proof}

\smallskip
\begin{remark}
The Hurwitz equivalence in the statement of Corollary~\ref{T:equivalence pencils} is not the classical one for fibrations, it is our (extended) Hurwitz equivalence for monodromy factorizations in the framed mapping class group, allowing the exceptional sections\,/\,base points to be interchanged. Although we believe the above corollary to be known to experts, we are not aware of any proof of it in the literature.
\end{remark}

\vspace{0.2in}
\section{Lefschetz fibrations which do not arise from pencils}\label{Stipsicz}

Although every Lefschetz pencil gives rise to a Lefschetz fibration on a blow-up of its total space, the converse is known to be false. As shown by Stipsicz \cite{Stipsicz}, and independently by Smith \cite{Smith}, if $(X,f)$ is a fiber sum of two nontrivial Lefschetz fibrations,\footnote{That is, the monodromy factorization $W_{X,f}$, up to Hurwitz equivalence, can be expressed as a product of two nontrivial positive factorizations of $1$.} it cannot have any exceptional sections, i.e. sections of self-intersection $-1$, and thus it is not a blow-up of a pencil. Motivated by this, Stipsicz conjectured in the same article that every Lefschetz fibration $(X,f)$, if not a blow-up of a pencil, is a fiber sum of such, which amounts to having Lefschetz pencils as building blocks of any Lefschetz fibration via fiber sums. 
 
In \cite{Sato_2008}, Sato proved that an interesting genus--$2$ Lefschetz fibration constructed by Auroux in \cite{Auroux_2003}, which could not be a fiber sum of nontrivial Lefschetz fibrations, did not have any exceptional sections either. This remained as the only known counter-example until recently, where in \cite{BaykurHayano_multisection}, we obtained several other genus--$2$ and $3$ counter-examples. The purpose of this section is to demonstrate the recipe of \cite{BaykurHayano_multisection} to generate such examples. We will do this while producing quick examples from a well-known relation in the mapping class group. We will then show that Auorux's example can also be derived in this very scheme. For various background results that goes into this recipe, we advise the reader to turn to \cite{BaykurHayano_multisection}.

\smallskip
\subsection{Examples derived by monodromy substitutions} \

Let $c_1, \ldots, c_5$ be the simple closed curves on $\Sigma_2^2$ as shown in Figure~\ref{scc_genus2surface}, and $\delta_1, \delta_2$ denote the two boundary components with marked points $u_1, u_2$. The chain relation of length $5$
\[ (t_{c_1}t_{c_2}t_{c_3}t_{c_4}t_{c_5})^6=t_{\delta_1} t_{\delta_2} \]
in  $\Mod(\Sigma_2;\{u_1,u_2\})$  (see \cite[Proposition 4.12]{Farb_Margalit_2011}), prescribes a triple $(X_0, f_0, S_0)$, which is a genus--$2$ Lefschetz fibration $(X_0, f_0)$ with a pure $2$--section $S$ that consists of two exceptional sections $S_1, S_2$. It is well-known that $X_0$ is the $\K$ surface blown-up twice,
the symplectic canonical class of which is represented by $[S]=[S_1] + [S_2]$ in $H_2(X_0; \Z)$.

We will need the following \emph{braiding lantern relation}, which is a generalization of the  lantern relation in the framed mapping class group:
\begin{lemma}\cite{BaykurHayano_multisection} \label{lem_lantern_relation}
Let the curves $a,b,c,d,x, \delta_1, \delta_2$, pairs of arcs $y,z$ and points $u_1, u_2$ in $\Sigma_0^6$ be as shown in Figure~\ref{fig_lantern_relation}, where $a,b,c,d,\delta_1,\delta_2$ are parallel to boundary components. Denote the boundary components parallel to $\delta_i$ by $S_i$. Then the relation \, $\widetilde{t_z} t_{x} \widetilde{t_{y}} = t_a t_b t_c t_d t_{\delta_2}$ \,  holds in $\Mod_{\partial \Sigma_0^6\setminus (S_1\sqcup S_2)}(\Sigma_0^6; \{u_1,u_2\})$ (whose elements restrict to identity along the four boundary components without marked points).

\begin{figure}[htbp]
\begin{center}
\includegraphics[width=35mm]{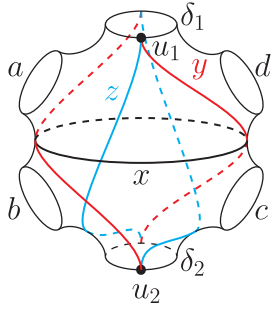}
\end{center}
\caption{Curves in $\Sigma_0^6$. }
\label{fig_lantern_relation}
\end{figure}
\end{lemma}

Substituting the subword on the right hand side of the above relation in the framed mapping class group (of a surface that contains the above subsurface with marked boundary components) with the subword on the left is then called a \emph{braiding lantern substitution}. 
Its importance and relevance to our current discussion is due to our observation in  \cite{BaykurHayano_multisection} that whenever the two marked boundary components  correspond to two exceptional sections $S_1, S_2$, we get a new exceptional $2$--section $S_{12}$ after the substitution, which we view as a result of \emph{braiding} $S_1$ and $S_2$ together. 

Forgetting the two marked boundary components, one gets the usual lantern relation. In this case, the subword $t_a t_b t_c t_d$ indicates that by clustering the corresponding Lefschetz critical points on the same singular fiber, we can obtain a  fiber \emph{component} $V$, which is a sphere of self-intersection $-4$. This $V$, which we will call a \emph{lantern sphere}, can be assumed to be symplectic with respect to a Gompf-Thurston form. Remembering the two marked boundary components, we conclude that the exceptional classes $S_1$ and $S_2$ each intersect $V$ positively at one point. 

Now, one can easily find a lantern sphere $V$ in the monodromy factorization of $(X_0, f_0)$ hit once by each exceptional section $S_i$, $i=1,2$. Remarkably, this holds for \emph{any} lantern sphere (and there are many of them; at least six disjoint ones~\cite{AkhmedovPark}). This is because for $V$ symplectic, the adjunction equality implies that its intersection with the canonical class $[S_1]+[S_2]$ is $2$, whereas each $S_i$ intersects the fibers positively.

Applying the braiding lantern substitution results in a new triple $(X_2, f_2, S_{12})$, a genus--$2$ Lefschetz fibration with an exceptional $2$--section $S_{12}$. 
As discussed in \cite{BaykurHayano_multisection}, an observation of Gompf shows that $X_2$ is diffeomorphic to an ordinary blow-down of $X_0$, so $X_2= \K \# \CPb$, which has only one exceptional class, already represented by $S_{12}$. (See \cite[Corollary 3]{Li_1999}. Note that $\K \# \CPb$ is not rational nor ruled.) 
It follows that $(X_2, f_2)$ does not have any other exceptional \emph{sections}. On the other hand, it was shown by Usher in \cite{Usher} that a nonminimal symplectic $4$--manifold cannot be a nontrivial fiber sum (also see \cite{BaykurDecomp} for a simpler proof for Lefschetz fibrations), so the presence of $S_{12}$ also implies that $(X_2, f_2)$ cannot be a fiber sum of any two nontrivial Lefschetz fibrations. Hence, we have obtained another example of a Lefschetz fibration which cannot arise from Lefschetz pencils via blow-ups and fiber sums.

\smallskip
\subsection{Auroux's genus--$2$ fibration with an exceptional $2$--section} \

Let $\delta_1$ and $\delta_2$ be simple closed curves in $\Sigma_2^2$ parallel to the boundary components containing $u_1$ and $u_2$, respectively. 
We take non-separating simple closed curves $c_1, c_2, c_3, c_4, c_5\subset\Sigma_2^2$, a path $\gamma\subset \Sigma_2^2$ between $u_1$ and $u_2$ and a pair $\sigma\subset\Sigma_2^2$ of paths connecting the two boundary components as shown in Figure~\ref{scc_genus2surface}. 
\begin{figure}[htbp]
\begin{center}
\includegraphics[width=70mm]{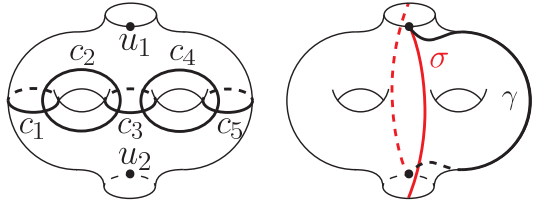}
\end{center}
\caption{Simple closed curves and paths in $\Sigma_2^2$. }
\label{scc_genus2surface}
\end{figure}
We will denote the right-handed Dehn twist along $c_i$ by $t_{i}\in \Mod(\Sigma_2^2)$. 

\begin{proposition}\label{T:lift_Auroux fibration}
The following relation holds in $\Mod(\Sigma_2^2;\{u_1,u_2\})$: 
\[
\tau_{\gamma}t_5t_4t_3t_2t_1t_1t_2t_3t_4t_5(t_3t_2t_1t_4t_3t_2t_5t_4t_3)^2\widetilde{t_\sigma} = t_{\delta_1}^2t_{\delta_2}^2. 
\]
\end{proposition}

\begin{proof}
Let $L = t_5t_4t_3t_2t_1t_1t_2t_3t_4t_5$ and $T = t_3t_2t_1t_4t_3t_2t_5t_4t_3$. 
It is easy to see (by the Alexander method (see \cite[\S.~2.3]{Farb_Margalit_2011}), for example) that $\widetilde{t_{\sigma}}$ is equal to $(t_2t_1)^3 (t_5t_4)^3 [\iota]$, where $[\iota]$ is the mapping class of the hyperelliptic involution $\iota$ given in Figure~\ref{loop_in_quotientsurface}. 
\begin{figure}[htbp]
\begin{center}
\includegraphics[width=90mm]{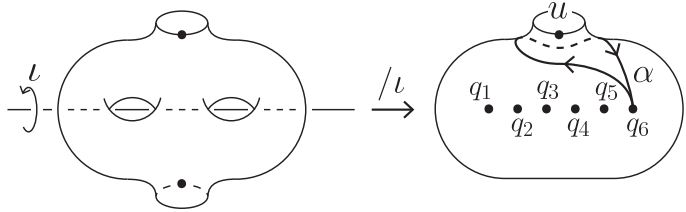}
\end{center}
\caption{Left~:~the hyperelliptic involution $\iota$. Right~:~the quotient surface $\Sigma_2^2/\iota$. }
\label{loop_in_quotientsurface}
\end{figure}
Thus, the product $T^2 \widetilde{t_\sigma}$ is calculated as follows: 
{\allowdisplaybreaks
\begin{align*}
T^2  \widetilde{t_\sigma} = & T t_3t_2t_1t_4t_3t_2t_5t_4t_3 (t_2t_1)^3(t_5t_4)^3 [\iota] \\
= & T t_3t_2t_1t_4t_3t_2(t_5t_4t_3t_2t_1) (t_2t_1)^2(t_5t_4)^3 [\iota] \\
= & T t_3t_2t_1(t_5t_4t_3t_2t_1)t_5t_4t_3(t_2t_1)^2(t_5t_4)^3 [\iota] \\
= & T t_3t_2t_1(t_5t_4t_3t_2t_1)^2t_2t_1(t_5t_4)^3 [\iota] \\
= & T (t_5t_4t_3t_2t_1)^3(t_5t_4)^3 [\iota] \\
= & T (t_2t_1)^3(t_5t_4t_3t_2t_1)^3 [\iota] \\
= & (t_5t_4t_3t_2t_1)^6 [\iota] \\
= & t_{\delta_1}t_{\delta_2} [\iota], 
\end{align*}}
where the last equality follows from the chain relation of length $5$. 
We can naturally regard the mapping class $L$ as an element in $\pi_0(C_{\emptyset}(\Sigma_2^2,\{u_1,u_2\};\iota))$, where $C_{\emptyset}(\Sigma_2^2,\{u_1,u_2\};\iota)$ is defined in \cite[\S.3.1]{BaykurHayano_multisection}. 
By \cite[Lemma 3.1]{BaykurHayano_multisection}, the kernel of the homomorphism 
\[
\iota_\ast :\pi_0(C_{\emptyset}(\Sigma_2^2,\{u_1,u_2\};\iota))\to \Mod(\Sigma_2^2/\iota;u,\{q_1,q_2,q_3,q_4,q_5,q_6\}) 
\]
induced by the quotient map $\Sigma_2^2\to \Sigma_2^2/\iota$ is generated by the class $[\iota]$, where $u$ denotes the point $[u_1]=[u_2] \in \Sigma_2^2/\iota$ and $\Mod(\Sigma_2^2/\iota;u,\{q_1,q_2,q_3,q_4,q_5,q_6\})$ consists of orientation preserving diffeomorphisms fixing $u$ (resp.~$\{q_1,q_2,q_3,q_4,q_5,q_6\}$) pointwise (resp.~setwise) modulo isotopies fixing the same data. 
It is easy to see that the image $\iota_\ast(L)$ is equal to the pushing map $\Push(\alpha)$ along the loop $\alpha$ in Figure~\ref{loop_in_quotientsurface}. 
The mapping class $\tau_\gamma^{-1}t_{\delta_1}t_{\delta_2}$ can also be regarded as an element in $\pi_0(C_{\emptyset}(\Sigma_2^2,\{u_1,u_2\};\iota))$, and it is sent to $\Push(\alpha)$ by $\iota_\ast$. 
Thus, $L^{-1}\tau_\gamma^{-1}t_{\delta_1}t_{\delta_2}$ is contained in the kernel of $\iota_\ast$. 
Since $L^{-1}\tau_\gamma^{-1}t_{\delta_1}t_{\delta_2}$ interchanges the points $u_1$ and $u_2$, this is equal to $[\iota]$. 
In particular the following relation holds in $\Mod(\Sigma_2^2; \{u_1,u_2\})$: 
\[
L = \tau_{\gamma}^{-1}t_{\delta_1} t_{\delta_2} [\iota]. 
\]
Thus, the product $\tau_{\gamma} LT^2 \widetilde{t_{\sigma}}$ is calculated as follows: 
{\allowdisplaybreaks
\begin{align*}
\tau_{\gamma}LT^2 \widetilde{t_{\sigma}} = & \tau_{\gamma}\tau_{\gamma}^{-1}t_{\delta_1} t_{\delta_2} [\iota] t_{\delta_1}t_{\delta_2} [\iota] \\
=& t_{\delta_1} t_{\delta_2} t_{\delta_2}t_{\delta_1} [\iota] [\iota] \\
=& t_{\delta_1}^2 t_{\delta_2}^2.
\end{align*}
}
This completes the proof. 
\end{proof}

As we explained in Section~\ref{S:factorization multisection}, we can regard the surface $\Sigma_2^2$ as a subsurface of $\Sigma_2$ by capping $\Pa\Sigma_2^2$ by two disks with the centers $s_1, s_2 \in\Sigma_2$. The pair of paths $\sigma$ and the path $\gamma$ respectively give rise to a simple closed curve in $\Sigma_2$ going through $s_1$ and $s_2$ and a simple path between $s_1$ and $s_2$. To simplify the notation,  we use the same symbols $\sigma$ and $\gamma$ to represent these curves. We also denote the Dehn twist along $c_i\subset \Sigma_2$ by $t_i \in \Mod(\Sigma_2;\{s_1,s_2\})$.  

Now by Proposition~\ref{T:lift_Auroux fibration}, we obtain the following  factorization in $\Mod(\Sigma_2;\{s_1,s_2\})$: 
\begin{equation}\label{E:factorization Auroux fibration in pointed MCG}
\tau_{\gamma}t_5t_4t_3t_2t_1t_1t_2t_3t_4t_5(t_3t_2t_1t_4t_3t_2t_5t_4t_3)^2\widetilde{t_\sigma} =1 \, ,
\end{equation}
which prescribes a triple $(X_1, f_1, S_1)$ where $f_1:X_1\to S^2$ is a genus--$2$ Lefschetz fibration with a sphere $2$--section $S_1$ by Theorem~\ref{MainQuote}. Under the forgetful homomorphism $\Mod(\Sigma_2;\{s_1,s_2\}) \to \Mod(\Sigma_2)$ this positive factorization 
maps to the monodromy factorization of Auroux's aforementioned genus--$2$ Lefschetz fibration given in \cite{Auroux_2003}. On the other hand, we can calculate the self-intersection number of $S_1$ using the positive factorization in the framed mapping class group we had in Proposition~\ref{T:lift_Auroux fibration} (which of course is a lift of the monodromy factorization $W_{X_1, f_1, S_1}!)$ and Theorem~1.1 of \cite{BaykurHayano_multisection}. Hence $(X_1, f_1, S_1)$ is Auroux's genus--$2$ Lefschetz fibration, where $S_1$ is the exceptional $2$--section.

\smallskip
\smallskip
We are now ready to show that $(X_1, f_1, S_1)$ can be reproduced using our recipe discussed in the previous subsection. Let $(X_0,f_0,S_0)$ denote the genus--$2$ Lefschetz fibration with a pure $2$--section $S_0$, which is a disjoint union of two exceptional sections, as prescribed by the $5$--chain relation  $(t_1t_2t_3t_4t_5)^6=1$ in $\Mod(\Sigma_2;\{s_1,s_2\})$. 

\smallskip
\begin{proposition}\label{T:Auroux fibration and chain fibration}
The triple $(X_1,f_1,S_1)$, where $(X_1,f_1)$ is Auroux's genus--$2$ fibration with the exceptional $2$--section $S_1$, is equivalent to a triple obtained from $(X_0,f_0,S_0)$ by a single braiding lantern substitution, followed by a perturbation of the $2$--section.
\end{proposition}

\begin{proof}
We first prove that the factorization $(t_1t_2t_3t_4t_5)^6=1$ in $\Mod(\Sigma_2;\{s_1,s_2\})$ is Hurwitz equivalent to the following factorization: 
\begin{equation}\label{E:intermediate factorization}
t_3t_2t_1t_4t_3t_2t_5t_4t_3t_5t_4t_3t_2t_1t_2t_3t_4t_5t_3t_2t_1t_4t_3t_2t_5t_4\underline{t_1t_1t_5t_5} =1. 
\end{equation}
The factorization \eqref{E:intermediate factorization} can be changed by elementary transformations as follows (in each line elementary transformations are applied to the underlined part to obtain the next line):
{\allowdisplaybreaks
\begin{align*}
& t_3t_2t_1t_4t_3t_2t_5t_4t_3t_5t_4t_3t_2t_1t_2t_3t_4t_5t_3t_2t_1t_4t_3t_2\underline{t_5t_4t_1t_1t_5}t_5 \\
\sim & t_3t_2t_1t_4t_3t_2t_5t_4t_3t_5t_4t_3t_2t_1t_2t_3t_4t_5t_3t_2t_1\underline{t_4t_3t_2t_4}t_5t_4t_1t_1t_5 \\ 
\sim & t_3t_2t_1t_4t_3t_2t_5t_4t_3t_5t_4t_3t_2t_1t_2t_3t_4t_5\underline{t_3t_2t_1t_3}t_4t_3t_2t_5t_4t_1t_1t_5 \\ 
\sim & t_3t_2t_1t_4t_3t_2t_5t_4t_3t_5t_4t_3\underline{t_2t_1t_2t_3t_4t_5}t_2t_3\underline{t_2t_1t_4t_3t_2t_5}t_4t_1t_1t_5 \\ 
\sim & t_3t_2t_1t_4t_3t_2t_5t_4t_3t_5t_4\underline{t_3t_1t_2t_3t_4t_5t_1t_2t_3t_4t_5}t_2\underline{t_1t_3t_2t_4t_1}t_1t_5 \\ 
\sim & t_3t_2t_1t_4t_3t_2t_5t_4t_3t_5t_4(t_1t_2t_3t_4t_5)^2t_1t_2t_3\underline{t_2t_1t_2t_4t_1t_5} \\ 
\sim & t_3t_2t_1t_4t_3t_2t_5t_4t_3t_5t_4\underline{(t_1t_2t_3t_4t_5)^2t_1t_2t_3t_4t_5t_2t_1t_2t_1} \\ 
\sim & \underline{t_3t_2t_1t_4t_3t_2t_5t_4t_3t_5t_4t_5t_4t_5t_4}(t_1t_2t_3t_4t_5)^3 \\ 
\sim & \underline{t_2t_1t_2}t_1t_2\underline{t_1t_3t_2t_1}t_4t_3t_2t_5t_4t_3(t_1t_2t_3t_4t_5)^3 \\ 
\sim & t_1t_2t_1t_1\underline{t_2t_3t_2}t_1\underline{t_2t_4t_3t_2}t_5t_4t_3(t_1t_2t_3t_4t_5)^3 \\ 
\sim & t_1t_2\underline{t_1t_1t_3}t_2\underline{t_3t_1t_4t_3}t_2\underline{t_3t_5t_4t_3}(t_1t_2t_3t_4t_5)^3 \\ 
\sim & t_1t_2t_3\underline{t_1t_1t_2t_1t_4}t_3\underline{t_4t_2t_5t_4}t_3t_4(t_1t_2t_3t_4t_5)^3 \\ 
\sim & t_1t_2t_3t_4\underline{t_1t_1t_2t_1t_3t_2t_5}t_4t_5t_3t_4(t_1t_2t_3t_4t_5)^3 \\ 
\sim & t_1t_2t_3t_4t_5t_1\underline{t_1t_2t_1}t_3t_2t_4t_5t_3t_4(t_1t_2t_3t_4t_5)^3 \\ 
\sim & t_1t_2t_3t_4t_5t_1t_2t_1\underline{t_2t_3t_2}t_4t_5t_3t_4(t_1t_2t_3t_4t_5)^3 \\ 
\sim & t_1t_2t_3t_4t_5t_1t_2\underline{t_1t_3}t_2t_3t_4t_5t_3t_4(t_1t_2t_3t_4t_5)^3 \\ 
\sim & t_1t_2t_3t_4t_5t_1t_2t_3\underline{t_1t_2t_3t_4t_5}t_3t_4(t_1t_2t_3t_4t_5)^3 \\ 
\sim & t_1t_2t_3t_4t_5t_1t_2t_3t_4t_5t_1t_2t_3t_4t_5(t_1t_2t_3t_4t_5)^3 \\ 
= & (t_1t_2t_3t_4t_5)^6.  
\end{align*}
}

We take pairs of paths $\xi, \zeta$ and a simple closed curve $a$ as shown in Figure~\ref{F:scc genus2surface2}. 
\begin{figure}[htbp]
\centering
\includegraphics[width=37mm]{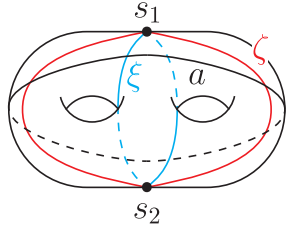}
\caption{Pair of paths and a simple closed curve in $\Sigma_2$. }
\label{F:scc genus2surface2}
\end{figure}
By Lemma~\ref{lem_lantern_relation}, $t_1t_1t_5t_5$ is equal to $\widetilde{t_\xi}t_3\widetilde{t_{\sigma}}$. 
Applying the substitution by this relation to the underlined part in \eqref{E:intermediate factorization}, we obtain the following factorization: 
\begin{align}\label{E:intermediate factorization2}
& t_3t_2t_1t_4t_3t_2t_5t_4t_3t_5t_4t_3t_2t_1t_2t_3t_4t_5\underline{t_3t_2t_1t_4t_3t_2t_5t_4\widetilde{t_\xi}}t_3\widetilde{t_{\sigma}}   \nonumber\\
\sim & t_3t_2t_1t_4t_3t_2t_5t_4t_3t_5t_4t_3t_2t_1t_2t_3t_4t_5\widetilde{t_{\zeta}}t_3t_2t_1t_4t_3t_2t_5t_4t_3\widetilde{t_{\sigma}},  
\end{align}
where the equivalence above holds since $t_3t_2t_1t_4t_3t_2t_5t_4t_3(\xi)$ is equal to $\zeta$.

According to \cite[Remark 3.9]{BaykurHayano_multisection}, we can perturb the $2$--section of the fibration corresponding \eqref{E:intermediate factorization2} so that it is away from the Lefschetz critical point corresponding $\widetilde{t_{\zeta}}$, and the resulting fibration with a $2$--section has the following factorization: 
\begin{align}\label{E:intermediate factorization3}
& t_3t_2t_1t_4t_3t_2t_5t_4t_3t_5t_4t_3t_2t_1\underline{t_2t_3t_4t_5t_a}\tau_{\gamma}t_3t_2t_1t_4t_3t_2t_5t_4t_3\widetilde{t_{\sigma}}\nonumber \\  
\sim & \underline{t_3t_2t_1t_4t_3t_2t_5t_4t_3t_5t_4t_3t_2t_1t_1t_2t_3t_4t_5\tau_{\gamma}}t_3t_2t_1t_4t_3t_2t_5t_4t_3\widetilde{t_{\sigma}} \nonumber\\  
\sim & \underline{t_5t_4t_3t_2t_1t_1t_2t_3t_4t_5\tau_{\gamma}}(t_3t_2t_1t_4t_3t_2t_5t_4t_3)^2\widetilde{t_{\sigma}} \nonumber \\
\sim & \tau_{\gamma}t_5t_4t_3t_2t_1t_1t_2t_3t_4t_5(t_3t_2t_1t_4t_3t_2t_5t_4t_3)^2\widetilde{t_{\sigma}} \nonumber. 
\end{align}
The last factorization coincides with the monodromy factorization of $(X_1, f_1, S_1)$ in the Equation~\eqref{E:factorization Auroux fibration in pointed MCG} above. 

We have thus derived $(X_1, f_1, S_1)$ from $(X_0, f_0, S_0)$, after a braiding lantern substitution, followed by a perturbation of the resulting $2$--section to move one of its branched points off the Lefschetz critical locus ---while applying Hurwitz moves for Lefschetz fibrations with multisections in various steps of the proof.
\end{proof}

\vspace{0.2in}


\end{document}